\documentclass{amsart}
\usepackage{amsthm,amsmath,amsfonts,amssymb,color, enumerate,tikz,hyperref,multicol,amscd,xypic,soul,witharrows}
\usepackage[all]{xy}

\newcommand{\h}{\mathcal{H}}
\newcommand{\ch}{\mathcal{CD}}
\newcommand{\pc}{\mathcal{PCD}}
\newcommand{\n}[1]{{}_{\nu}#1}
\newcommand{\N}{\mathbb{N}}
\newcommand{\Z}{\mathbb{Z}}

\newcommand{\hb}[1]{\mathcal{H}\left(p^{#1}\right)}

\newtheorem{Theorem}{Theorem}[section]
\theoremstyle{definition}
\newtheorem{Definition}[Theorem]{Definition}
\newtheorem{Notation}[Theorem]{Notation}
\newtheorem{Lemma}[Theorem]{Lemma}
\newtheorem{Corollary}[Theorem]{Corollary}

\newtheorem{Remark}[Theorem]{Remark}

\newtheorem{Example}[Theorem]{Example}

    \newtheoremstyle{TheoremNum}
        {\topsep}{\topsep}              
        {\itshape}                      
        {}                              
        {\bfseries}                     
        {.}                             
        { }                             
        {\thmname{#1}\thmnote{ \bfseries #3}}
    \theoremstyle{TheoremNum}
 \newtheorem{thmn}{Theorem}

  \newtheoremstyle{LemmaNum}
        {\topsep}{\topsep}              
        {\itshape}                      
        {}                              
        {\bfseries}                     
        {.}                             
        { }                             
        {\thmname{#1}\thmnote{ \bfseries #3}}
    \theoremstyle{LemmamNum}

\CompileMatrices

\begin{document}

\title[The CD Lattice of the $\mathbb{Z}_{p^n}$ Heisenberg Group]{ Pseudocentralizers and The Chermak-Delgado Measure of the mod $p^{n}$ Heisenberg Group}

\author{David Allen}
\address{Department of Mathematics, Borough of Manhattan Community College (CUNY), New York 10007}
\email{dtallen@bmcc.cuny.edu}

\author{Jos\'{e} J. La Luz}
\address{Departmento de Matem\'aticas, Universidad de Puerto Rico,
Industrial Minillas 170 Car 174, Bayam\'on, PR, 00959-1919}
\email{jose.laluz1@upr.edu}

\author{Stephen Majewicz}
\address{Department of Mathematics, Kingsborough Community College (CUNY), Brooklyn, New York 11235}
\email{smajewicz@kbcc.cuny.edu}

\author{Marcos Zyman}
\address{Department of Mathematics, Borough of Manhattan Community College (CUNY), New York 10007}
\email{mzyman@bmcc.cuny.edu}


\begin{abstract}
In this paper we compute the Chermak-Delgado measure of the mod $p^{n}$ Heisenberg Group for any prime $p.$ To achieve this we introduce the notion of the pseudocentralizer and prove various results about it.
\end{abstract}

\maketitle

\section{Introduction}

Let $G$ be a finite group. The \emph{Chermak-Delgado measure} of a subgroup $H$ of $G,$ denoted by $m(H),$ is defined by $m(H) = |H||C(H)|.$ The maximum of the set $\{m(H) \, | \, H \le G\}$ is termed the \emph{Chermak-Delgado measure} of $G$ and is denoted by $m^{\ast}(G).$ In \cite{CD}, A. Chermak and A. Delgado defined these measures and established a certain sublattice of the lattice of all subgroups of $G.$ This sublattice has been studied in many instances. In particular, certain properties of subgroups of finite simple groups have been established using it.

Various papers in the literature deal with theoretical implications of the concept (for example, see \cite{BW} for the relation between the Chermak-Delgado lattice of a group and its subnormal subgroups and \cite{W} for some basic calculations). The purpose of this paper is to obtain an explicit calculation of the Chermak-Delgado measure for $\h(p^{n}),$ the Heisenberg group with entries in the ring $\mathbb{Z}_{p^n}$ for any prime $p$ and $n > 0.$ To achieve this we resort to some simple homological techniques.

We begin by introducing a new notion called the pseudocentralizer and developing some of its properties in order to obtain our computational results. If $G$ and $K$ are groups and $q : G \rightarrow K$ is a homomorphism, then the \emph{$q$-pseudocentralizer} of a non-empty subset $S$ of $G$ is
\[
P(S; \, q) = \{g \in G \, | \, [s, \, g] \in \ker q \hbox{ for all } s \in S\}.
\]
If $S = \{s\},$ then we write $P(s; \, q).$ For simplicity, we write $P(S)$ for $P(S; \, q)$ when $q$ is fixed and call $P(S)$ the \emph{pseudocentralizer} of $S.$ For our purposes, the homomorphism $q$ will be from the short exact sequence of groups
\begin{center}
\ \ \ \ \ \ \ \ \ \ \xymatrix{1 \ar[r] & \mathbb{Z}_{p}^{3} \ar[r]^{f} & \h(p^{n}) \ar[r]^{q} & \h\left(p^{n - 1}\right) \ar[r] & 1}
\end{center}

\noindent where $f\left(a_{1}, \, a_{2}, \, a_{3}\right) = \left(a_{1}p^{n - 1}, \, a_{2}p^{n - 1}, \, a_{3}p^{n - 1}\right)$ and

\ \ \ \ \ \ $q\left(b_{1}, \, b_{2}, \, b_{3}\right) = \left(b_{1} \hbox{ mod }p^{n - 1}, \, b_{2} \hbox{ mod }p^{n - 1}, \, b_{3} \hbox{ mod }p^{n - 1}\right)$ where the elements of the Heisenberg groups are naturally expressed using vector notation. 

\vspace{.1in}

The reason for introducing the pseudocentralizer of a set $S$ is to study its relation to $C(S),$ the centralizer of $S.$ After showing that $C(S) \unlhd P(S)$  for any $S \subseteq \h(p^{n})$ in the short exact sequence above (see Lemma~\ref{CnormalinP}), we prove the first important theorem in this direction:

\begin{thmn}[\ref{lo2}]
If $h$ is a non-central element of $\hb{n},$ then $|P(h)| = p|C(h)|$. Hence, $P(h)/C(h) \cong \mathbb{Z}_{p}.$
\end{thmn}

Using this theorem, we prove that if $H \le \h(p^{n})$ and $H$ is not in the center of $\h(p^{n}),$ then $P(H)/C(H)$ is isomorphic to either $\mathbb{Z}_{p}$ or $\mathbb{Z}_{p} \times \mathbb{Z}_{p}$  (see Lemma~\ref{diagramChasing}). It turns out that under certain conditions, we know exactly when $P(H)/C(H)$ is isomorphic to $\mathbb{Z}_{p} \times \mathbb{Z}_{p},$ and thus, when $|P(H)| = p^{2}|C(H)|.$ This is the purpose of Theorem~\ref{main1}.

After introducing the concept of an \emph{injective set} for a \emph{special} generating set of $H,$ we generalize Theorem~\ref{main1}. More precisely, we prove:

\begin{thmn}[\ref{main}]
Let $S$ be a special generating set for $H \le \hb{n}$ and $I$ an injective set for $S.$
\begin{enumerate}
\item If $|I| = 0,$ then $|P(H)| = |C(H)|.$
\item If $|I| = 1,$ then $|P(H)| = p|C(H)|.$
\item If $|I| = 2,$ then $|P(H)| = p^{2}|C(H)|.$
\end{enumerate}
\end{thmn}

This theorem plays a major role in proving various computational results about the the Chermak-Delgado measure of subgroups of $\h(p^{n}).$ In addition, we define and study an analogue of the Chermak-Delgado measure termed the \emph{pseudo Chermak-Delgado measure}. We provide numerous results in Section 7 that relate the two notions. These lead us to the main theorem in this paper:

\begin{thmn}[\ref{MAIN}]
For any prime $p,$ we have $m^{\ast}(\hb{n}) = p^{4n}.$
\end{thmn}

\section{Notation and Basic Facts}

In this section, we provide some notation and background material that will be used in the paper. We assume throughout that all groups are finite. Let $G$ be such a group.

\begin{itemize}
\item The centralizer of $a \in G,$ commonly written as $C_{G}(a),$ is denoted by $C(a)$ ($G$ will always be understood from the context). Similarly, the centralizer of $H \leq G$ is written as $C(H)$ rather than the usual $C_{G}(H).$
\item $Z(G)$ is the center of $G.$
\item $|G|$ is the order of $G.$  For $g \in G,$ we write $o(g)$ for its order.
\item $G^{n}$ is the external direct product $G \times \cdots \times G$ with $n$ direct factors.
\item If $S = \{g_{1}, \, g_{2}, \, \ldots, \, g_{k}\} \subseteq G,$ then $\left< S \right> = \left<g_{1}, \, g_{2}, \, \ldots, \, g_{k}\right>$ denotes the subgroup of $G$ generated by $S.$
\item The commutator of $x, \, y \in G$ is $[x, \, y] = x^{-1}y^{-1}xy$ and the conjugate of $x$ by $y$ is $x^{y} = y^{-1}xy.$
\item $[G, \, G] = \left<[g, \, h] \, | \, g, \, h \in G\right>$ is the derived subgroup of $G.$
\item We write $\Z_{n}$ for the ring $(\Z_{n}, \ +, \ \cdot).$
\end{itemize}

Additional notations and definitions will be presented when they first appear in the paper.

\section{The $q$-pseudocentralizer}

In this section, we introduce the notion of the $q$-pseudocentralizer and provide some of its elementary properties.

\begin{Definition}\label{d:Pseudocentralizer}
Let $G$ and $K$ be groups, and suppose that $q : G \rightarrow K$ is a homomorphism. The \emph{$q$-pseudocentralizer} of a non-empty subset $S$ of $G$ is
\[
P(S; \, q) = \{g \in G \, | \, [s, \, g] \in \ker q \hbox{ for all } s \in S\}.
\]
\end{Definition}

If $S = \{s\},$ then we write $P(s; \, q).$ Clearly, $P(S; \, q) \neq \emptyset$ since it contains $1.$

\vspace{.15in}

We will be interested in the case when the homomorphism $q$ arises in a short exact sequence. Let $G, \, G_{1},$ and $G_{2}$ be groups. In the paper, we will refer to the short exact sequence:
\begin{equation}\label{shes}
\xymatrix{1 \ar[r] & G_{1} \ar[r]^{f} & G \ar[r]^{q} & G_{2} \ar[r] & 1}.
\end{equation}
This means that $f$ and $q$ are group homomorphisms, $f$ is injective, $q$ is surjective, and im $f = \ker q.$ Note that ``$1$" in the sequence means the trivial group $\{1\}.$

\begin{Lemma}
If $S \subseteq G$ and $H \lhd G$ in the short exact sequence (\ref{shes}), then we have $P(S; \, q) \leq G$ and $P(H; \, q) \lhd G.$
\end{Lemma}

\begin{proof}
Let $g_{1}, \, g_{2} \in P(S; \, q).$ Using the commutator identities (see Lemma 1.4 of \cite{CMZ}), we see that
\[
\Bigl[s, \, g_{1}g^{-1}_{2}\Bigr] = \Bigl[s, \, g_{2}^{-1}\Bigr]\Bigl[s, \, g_{1}\Bigr]^{g_{2}^{-1}} = \Bigl([s, \, g_{2}]^{g_{2}^{-1}}\Bigr)^{-1}[s, \, g_{1}]^{g^{-1}_{2}}.
\]
Since $[s, \, g_{1}] \in \ker q,$ $[s, \, g_{2}] \in \ker q,$ and $\ker q \unlhd G,$ we have that $[s, \, g_{1}]^{g^{-1}_{2}} \in \ker q$ and $[s, \, g_{2}]^{g_{2}^{-1}} \in \ker q.$ It follows that $P(S; \, q) \leq G.$

Next let $z \in P(H; \, q).$ If $g \in G$ and $h \in H,$ then $h^{g^{-1}} \in H$ because $H \lhd G.$ Hence, $\Bigl[h^{g^{-1}}, \, z\Bigr] \in \ker q$ and, consequently,
\[
q\Bigl(\Bigl[h, \, z^{g}\Bigr]\Bigr) = q\Bigl(\Bigl[h^{g^{-1}}, \, z\Bigr]^{g}\Bigr) = q\Bigl(\Bigl[h^{g^{-1}}, \, z\Bigr]\Bigr)^{q(g)} = 1.
\]
And so, $z^{g} \in P(H; \, q).$
\end{proof}

\begin{Notation}
Whenever the map $q$ is fixed, we write $P(S)$ for $P(S; \, q)$ and call $P(S)$ the \emph{pseudocentralizer} of $S.$ Whether or not $S$ is a subgroup
of $G$ will be apparent from the context.
\end{Notation}

\begin{Example}
Let $G$ be a group and $id : G \rightarrow G$ the identity homomorphism.

\noindent 1) For the short exact sequence $\xymatrix{1 \ar[r] & G \ar[r]^{id} & G \ar[r]^{q} & 1 \ar[r] & 1}$
we have $P(H) = G$ for any $H \le G.$\\

\noindent 2) For the short exact sequence $\xymatrix{1 \ar[r] & 1 \ar[r]^{f} & G \ar[r]^{id} & G \ar[r] & 1}$
we have $P(H) = C(H)$ for any $H \le G.$
\end{Example}

\begin{Lemma}\label{P-properties}
For any short exact sequence (\ref{shes}) with $H \le G$ and $K \le G,$ and any subset $S$ of $G,$ the following properties hold:
\begin{enumerate}
\item $C(S) \le P(S)$ for any $S\subseteq  G$;

\vspace{.05in}

\item $q(P(H)) = C(q(H))$;

\vspace{.05in}

\item $\ker q \le P(H)$;

\vspace{.05in}

\item $|P(H)| = |G_{1}||C(q(H))|$;

\vspace{.05in}

\item $\displaystyle P(S) = \bigcap_{s \in S} P(s)$;

\vspace{.05in}

\item $\displaystyle P(HK) = P(H) \cap P(K)$;

\vspace{.05in}

\item $P(HK) = P(\left<HK\right>)$;

\vspace{.05in}

\item $P(H) P(K) \le P(H \cap K)$;

\vspace{.05in}

\item If $S$ is a generating set for $H,$ then $P(H) = P(S)$;

\vspace{.05in}

\item If $S_{1} \subseteq S_{2} \subseteq G,$ then $P(S_{2}) \le P(S_{1}).$
\end{enumerate}
\end{Lemma}

The proof of Lemma~\ref{P-properties} is left for the reader. We remark that $(6)$ and $(7)$ hold even if $HK$ is not a subgroup of $G.$ The proof of (4) makes use of the next theorem which will be needed again later. Its proof can be found in the Appendix.

\begin{Theorem}\label{L1}
Suppose that $H \le G$ in the short exact sequence (\ref{shes}). If $\widehat{H}_{1} = f^{-1}(H \cap \ker q)$ and $\widehat{H}_{2} = q(H),$ then
$|H| = \left|\widehat{H}_{1}\right|\left|\widehat{H}_{2}\right|.$
\end{Theorem}

\begin{Lemma}\label{abelian-case}
If $G_{2}$ is abelian in the short exact sequence (\ref{shes}), then $P(H) = G$ for any $H \le G$.
\end{Lemma}

\begin{proof}
If $h \in H$ and $g \in G,$ then $q([h, \, g]) = [q(h), \, q(g)].$ Since $G_{2}$ is abelian, $[q(h), \, q(g)]= 1$ and, thus, $[h, \, g] \in \ker q.$ And so, $g \in P(H)$.
\end{proof}

\section{the mod $r$-Heisenberg group}

\noindent \textbf{From this point on, $r > 1$ will always be some natural number and $p$ is a fixed prime number.}

\vspace{.1in}

\begin{Definition}\label{d:HeisenbergGroup}
For any ring $\Z_{r},$
\[
\h(r) = \left\{
\begin{bmatrix}
1 & a_{1} & a_{2}\\
0 & 1 & a_{3}\\
0 & 0 & 1
\end{bmatrix} \, \middle| \, a_{1}, \, a_{2}, \, a_{3} \in \Z_{r}
\right\}
\]
is the \emph{mod $r$ Heisenberg group} with matrix multiplication as the group operation.
\end{Definition}

\vspace{.1in}

\begin{Notation}
The matrix $a = \begin{bmatrix}
1 & a_{1} & a_{2}\\
0 & 1 & a_{3}\\
0 & 0 & 1
\end{bmatrix}$
will be written in the vector form $a = (a_{1}, \, a_{2}, \, a_{3}).$ We call $a_{i}$ the $i^{th}$ \emph{component} of $a.$
\end{Notation}

\vspace{.1in}

Using this notation and standard matrix multiplication, we get
\begin{eqnarray*}
(a_{1}, \, a_{2}, \, a_{3})(b_{1}, \, b_{2}, \, b_{3}) & = & (a_{1} + b_{1}, \, a_{2} + b_{2} + a_{1}b_{3}, \, a_{3} + b_{3}) \hbox{ \ \ and } \\
                      (a_{1}, \, a_{2}, \, a_{3})^{-1} & = & (-a_{1}, \, a_{1}a_{3} - a_{2}, -a_{3}).
\end{eqnarray*}
In the above, all calculations are obviously done in the ring $\Z_{r}.$ Observe that the vector $(0, \, 0, \, 0)$ represents the multiplicative identity in $\h(r)$.

\begin{Notation}
The $i^{th}$ component of $g \in \h(r)$ will usually be written as $g|_{i}.$

\begin{itemize}
\item When we use this notation \textbf{outside} of a vector, we will write ``mod" to emphasize that we are calculating modulo $r$ in a component. Thus, for the vector $c = (a_{1} + b_{1}, \, a_{2} + b_{2} + a_{1}b_{3}, \, a_{3} + b_{3}) \hbox{ in } \h(r),$ we write
\[
\indent c|_{i} = (a_{i} + b_{i}) \hbox{ mod } r \hbox{ for } i = 1, \, 3 \hbox{ \ \ and \ \ } c|_{2} = (a_{2} + b_{2} + a_{1}b_{3}) \hbox{ mod } r.
\]
Note that we did not write ``$c|_{i} \hbox{ mod } r$".

\vspace{.05in}

\item Unless ambiguity arises, when this notation is used \textbf{inside} of a vector (see Lemma~\ref{known} for example), operations are understood to be modulo $r$ without mention and we will not write ``mod".
\end{itemize}
\end{Notation}

\vspace{.15in}

A simple calculation and induction on $n$ give:

\begin{Lemma}\label{known}
For any natural number $n > 1$ and $a \in \h(r),$ we have
\[
a^{n} = \left(na|_{1}, \, na|_{2} + {n\choose 2}a|_{1}a|_{3}, \, na|_{3}\right).
\]
\end{Lemma}

\vspace{.1in}

The proof of the next lemma is straightforward.

\begin{Lemma}\label{conmutador}
Suppose that $a, \, b \in \h(r).$
\begin{enumerate}
\item $[a, \, b] = (0, \, a|_{1}b|_{3} - a|_{3}b|_{1}, \, 0).$

\vspace{.05in}

\item $Z(\h(r)) = [\h(r), \, \h(r)] = \{(0, \, x, \, 0) \, | \, x \in \Z_{r}\}.$

\vspace{.05in}

\item $\h(r)$ is nilpotent of class $2.$
\end{enumerate}
\end{Lemma}

\begin{Corollary}\label{C}
Let $S \subseteq \h(r)$ and $h \in \h(r)$.
\begin{enumerate}
\item $C(S) = \{g \in \h(r) \, | \, [a, \, g]|_{2} = 0  \hbox{ mod } r\, \hbox{ for all } a \in S\};$
\item $C(h) = \{g \in \h(r) \, | \, [h, \, g]|_{2} = 0  \hbox{ mod } r \}.$
\end{enumerate}
\end{Corollary}

\begin{proof}
It follows from Lemma~\ref{conmutador}.
\end{proof}

\begin{Lemma}\label{producttosum}
Suppose that $x, \, y, \, z, \, w \in \h(r).$ The following hold:
\begin{enumerate}
\item $[[x, \, y], \, z] = (0, \, 0, \, 0)$;
\item $[x, \, yz] = [x, \, z][x, \, y]$ and $[xy, \, z] = [x, \, z][y, \, z]$;
\item $([x, \, y][z, \, w])|_{2} = [x, \, y]|_{2} + [z, \, w]|_{2}$;
\item $[x, \, y]|_{2} + \left[x, \, y^{-1}\right]|_{2} = 0$.
\end{enumerate}
\end{Lemma}

\begin{proof}
This is a consequence of the commutator identities (see Lemma 1.4 of \cite{CMZ}) and Lemma~\ref{conmutador}.
\end{proof}

We henceforth specialize to the case when $r = p^{n}$ for any fixed $n \in \mathbb{N}$ and prime $p.$ The next lemma provides a short exact sequence which will be referred to throughout the paper. The proof is left for the reader.

\begin{Lemma}\label{short}
There is a short exact sequence of groups
\begin{equation}\label{sh-H}
\xymatrix{1 \ar[r] & \mathbb{Z}_{p}^{3} \ar[r]^{f} & \h(p^{n}) \ar[r]^{q} & \h\left(p^{n - 1}\right) \ar[r] & 1}
\end{equation}
\noindent where $f\left(a_{1}, \, a_{2}, \, a_{3}\right) = \left(a_{1}p^{n - 1}, \, a_{2}p^{n - 1}, \, a_{3}p^{n - 1}\right)$ and

\ \ \ \ \ \ $q\left(b_{1}, \, b_{2}, \, b_{3}\right) = \left(b_{1} \hbox{ mod }p^{n - 1}, \, b_{2} \hbox{ mod }p^{n - 1}, \, b_{3} \hbox{ mod }p^{n - 1}\right).$
\end{Lemma}

\vspace{.1in}

\noindent \textbf{Unless otherwise told, the map $q$ in $P(S) = P(S; \, q)$ will be from the short exact sequence (\ref{sh-H}).}

\begin{Corollary}\label{P}
Let $S \subseteq \h(p^{n})$ and $h \in \h(p^{n}).$
\begin{enumerate}
\item $P(S) = \{g \in \h(p^{n}) \, | \, [a, \, g] |_{2} = kp^{n - 1} \hbox{ for some } 0 \le k < p \hbox{ and all } a \in S\}$.

\vspace{.05in}

\item $P(h) = \{g \in \h(p^{n}) \, | \, [h, \, g] |_{2} = kp^{n - 1} \hbox{ for some } 0 \le k < p\}$.
\end{enumerate}
\end{Corollary}

\begin{proof}
Since $\ker q = \{g \in \h(p^{n}) \, | \, p^{n - 1}\;\text{divides}\; g|_{i} \hbox{ for each } i = 1, \, 2, \, 3\},$ the result follows from Lemma \ref{conmutador}.
\end{proof}

\section{relationships between $P(H)$ and $C(H)$}

Let $H \leq \h(p^{n})$ and $h \in \h(p^{n}).$ In this section, we will prove a series of statements concerning the precise relation between $P(H)$ and $C(H),$ as well as between $P(h)$ and $C(h).$ Most of these relationships are in terms of the center of $\h(p^{n}).$ These, in turn, will allow us to determine how their orders are related.

\begin{Lemma}\label{CnormalinP}
For any $S\subseteq \h(p^{n})$ in the short exact sequence (\ref{sh-H}), we have $C(S) \unlhd P(S)$.
\end{Lemma}

Compare this to Lemma~\ref{P-properties} (1). Note that our proof actually shows that $C(H)$ and $C(h)$ are normal in $\h(p^{n})$.

\begin{proof}
Let $k \in H, \, c \in C(S),$ and $z \in P(S).$ By Lemma~\ref{producttosum}, $\left[k, \, z^{-1}cz\right]|_{2} = 0.$ Thus, $z^{-1}cz \in C(S)$ by Corollary~\ref{C}. And so, $C(S) \unlhd P(S).$ 
\end{proof}

\begin{Lemma}\label{PnoC2}
Let $H \leq \h(p^{n}).$ Then $H \leq Z(\h(p^{n}))$ if and only if $P(H) = C(H).$
\end{Lemma}

\begin{proof}
If $H \leq Z(\h(p^{n})),$ then $C(H) = \h(p^n).$ It follows from Lemma~\ref{P-properties} (1) that $P(H) = C(H).$ The proof of the converse will be given after Definition~\ref{d:InjectiveSetForSubgroup}.
\end{proof}

\begin{Lemma}\label{PnoC}
For all $h \in \h(p^{n}),$ $P(h) = C(h)$ if and only if $h \in Z(\h(p^{n})).$
\end{Lemma}

\begin{proof}
If $h \in Z(\h(p^{n})),$ then $C(h) = \h(p^{n}).$ By Lemma~\ref{P-properties} (1), $C(h) \leq P(h)$ and, thus, $P(h) = C(h).$

Suppose now that $h \notin Z(\h(p^{n})).$ Then either $h|_{1} \neq 0$ or $h|_{3} \neq 0.$ If $h|_{1}\ne 0,$ then there exist $0 \le k < n$ and $0 < r < p^{n}$ such that $p\nmid r$ and $h|_{1} = rp^{k}$. Let $g = \left(0, \, 0, \, p^{n - k - 1}\right) \in \h(p^{n}).$ We claim that $g \in P(h) \setminus C(h).$ Well, by Lemma~\ref{conmutador}, $[h, \, g] = \left(0, \, rp^{n - 1}, \, 0\right).$ Since $p^{n - 1}r \hbox{ mod } p^{n - 1} = 0,$ we have that $[h, \, g] \in \ker q.$ And so, $g \in P(h).$ Evidently, $g \notin C(h)$ by Corollary~\ref{C}. This proves the claim. The case $h|_{3} \ne 0$ is proven in a similar way.
\end{proof}

\vspace{.05in}

Non-central elements play an important role in what follows. We introduce some notation.

\begin{Notation}
If $S$ is a non-empty subset of $\hb{n},$ then we put
\[
S^{\dagger} = S - Z(\hb{n}).
\]
In particular, $\hb{n}^{\dagger}$ is the set of non-central elements of $\hb{n}.$
\end{Notation}

\begin{Lemma}\label{l:P(H)=P(SN)}
If $H \leq \hb{n}$ and $H = \left<S \right>$ for some non-empty set $S$ of $\hb{n},$ then $P(H) = P\left(S^{\dagger}\right)$ and $C(H) = C\left(S^{\dagger}\right)$.
\end{Lemma}

\begin{proof}
By (9) of Lemma \ref{P-properties} we have $P(H) = P(S)$. Hence,
\[
P(S) = \bigcap_{h \in S}P(h) = \left(\bigcap_{h \in S^{\dagger}}P(h)\right) \bigcap \left(\bigcap_{h \in S - S^{\dagger}}P(h)\right) = P\left(S^{\dagger}\right) \cap \hb{n} = P\left(S^{\dagger}\right).
\]
The proof is similar for $C(H)$.
\end{proof}

Next we determine the relationship between $|P(h)|$ and $|C(h)|$ for any $h \in \hb{n}^{\dagger}.$ In order to do so, we partition $P(h)$ into a certain collection of subsets. The following notation will be used:

\begin{Notation}
For any $h \in \hb{n}^{\dagger}$ and each $0 \leq \ell < p,$ let
\[
P_{\ell}(h) = \left\{g \in  \h(p^{n}) \, \vert \, [h, \, g]|_{2} = \ell p^{n - 1}\right\}  \subset P(h).
\]
Observe that $P_{0}(h) = C(h)$ by Corollary~\ref{C}.
\end{Notation}

\begin{Theorem}\label{lo2}
If $h \in \hb{n}^{\dagger},$ then $|P(h)| = p|C(h)|$. Hence, $P(h)/C(h) \cong \mathbb{Z}_{p}.$
\end{Theorem}

\begin{proof}
Since $h$ is not central, then either $h|_{1} \ne 0$ or $h|_{3} \ne 0.$ Assume first that $h|_{1} \ne 0.$ There exist $0 \le k < n$ and $0 < r < p^{n}$ such that $p \nmid r$ and $h|_{1} = rp^{k}.$ Note that $r$ is a unit in the ring $\Z_{p^{n}}.$

We concoct set maps between $C(h)$ and $P_{\ell}(h)$ for each $\ell.$ First observe that if $x \in C(h),$ then for each $\ell,$ Lemma~\ref{conmutador} and Corollary~\ref{C} give
\[
\left[h, \, \left(x|_{1}, \, x|_{2}, \, x|_{3} + r^{-1}\ell p^{n - k - 1}\right)\right]|_{2} = \ell p^{n - 1}.
\]
Thus, $\left(x|_{1}, \, x|_{2}, \, x|_{3} + r^{-1}\ell p^{n - k - 1}\right) \in P_{\ell}(h).$ Hence, there exists a well-defined set map $f_{\ell} : C(h) \rightarrow P_{\ell}(h)$ given by
\[
f_{\ell}(x) = \left(x|_{1}, \, x|_{2}, \, x|_{3} + r^{-1} \ell p^{n - k - 1} \right).
\]

Next observe that if $y \in P_{\ell}(h)$ for some $\ell,$ then $[h, \, y]|_{2} = \ell p^{n - 1}.$ Thus, using Lemma~\ref{conmutador}, we obtain
\[
\left[h, \, \left(y|_{1}, \, y|_{2}, \, y|_{3} - r^{-1}\ell p^{n - k - 1}\right) \right]|_{2} = 0.
\]
By Corollary~\ref{C}, $\left(y|_{1}, \, y|_{2}, \, y|_{3} - r^{-1}\ell p^{n - k - 1}\right) \in C(h).$ And so, there exists a well-defined set map $g_{\ell} : P_{\ell}(h) \rightarrow C(h)$ given by
\[
g_{\ell}(y) = \left(y|_{1}, \, y|_{2}, \, y|_{3} - r^{-1} \ell p^{n - k - 1}\right).
\]
Clearly, $f_{\ell}$ and $g_{\ell}$ are mutual inverses for each $0 \leq \ell < p.$

Similarly, if $h|_{3} \ne 0,$ then we may write $h|_{3} = rp^{k},$ where $0 \le k < n,$ $p \nmid r,$ and $0 < r < p^{n}$ (thus, $r$ is a unit in $\Z_{p^{n}}$).
In this case, we obtain well-defined set maps that are mutual inverses:
\[
\ \ \ \ \ \ \ \ f_{\ell} : C(h) \rightarrow P_{\ell}(h) \hbox{ \, given by \, }
x \mapsto \left(x|_{1} - r^{-1} \ell p^{n - k - 1}, \, x|_{2}, \, x|_{3}\right), \hbox{ \ \ and }
\]
\[
g_{\ell} : P_{\ell}(h) \rightarrow C(h) \hbox{ \, given by \, }
y \mapsto \left(y|_{1} + r^{-1} \ell p^{n - k - 1}, \, y|_{2}, \, y|_{3}\right).
\]
In either case, we have that $|P_{\ell}(h)| = |C(h)|$ for each $0 \leq \ell < p.$ Since $P(h) = \bigcup_{\ell = 0}^{p - 1} P_{\ell}(h)$ and the $P_{\ell}(h)$ are pairwise disjoint, we have
\[
|P(h)| = \sum_{\ell = 0}^{p - 1}|P_{\ell}(h)| = p|C(h)|.
\]
This completes the proof.
\end{proof}

\begin{Notation}\label{n:Factored_h}
From this point on, whenever we write $h \in \h(p^{n})$ in the form
\[
h = \Bigl(r_{1}p^{k_{1}}, \, r_{2}p^{k_{2}}, \, r_{3}p^{k_{3}}\Bigr),
\]
it is understood that either:

\vspace{.025in}

\noindent (1) \ $r_{i} = 0$ or

\vspace{.025in}

\noindent (2) \ $0 \leq k_{i} < n,$ $0 < r_{i} < p^{n},$ and $p \nmid r_{i}$

\vspace{.025in}

\noindent for $i = 1, \, 2, \, 3.$ Thus, $k_{i}$ is the largest power of $p$ that divides $h|_{i}$ whenever $r_{i} \neq 0.$
\end{Notation}

\vspace{.1in}

The functions $f_{\ell}$ in the proof of Theorem~\ref{lo2} allow us to express elements of $P_{\ell}(h)$ in terms of elements of $C(h).$ This important fact is the highlight of the next result which is needed later.

\begin{Lemma}\label{pForm}
Let $h = \Bigl(r_{1}p^{k_{1}}, \, r_{2}p^{k_{2}}, \, r_{3}p^{k_{3}}\Bigr) \in \h(p^{n}).$ If $x \in P_{\ell}(h)$ for some $0 \leq \ell < p,$ then there exists $(a_{1}, \, a_{2}, \, a_{3}) \in C(h)$ such that
\begin{eqnarray*}
x & = & \Bigl(a_{1}, \, a_{2}, \, a_{3} + r_{1}^{-1} \ell p^{n - k_{1} - 1}\Bigr)\qquad \text{if} \qquad r_{1} \ne 0 \hbox{ \ \ \ and}\\
x & = & \Bigl(a_{1} - r_{3}^{-1} \ell p^{n - k_{3} - 1}, \, a_{2}, \, a_{3}\Bigr)\qquad \text{if} \qquad r_{3} \ne 0.
\end{eqnarray*}
\end{Lemma}

\begin{proof}
This follows from the functions constructed in Lemma \ref{lo2}.
\end{proof}

\begin{Lemma}\label{premain1}
If $h \in \h(p^{n})^{\dagger},$ then $|C(h)| = p^{2}|C(q(h))|.$
\end{Lemma}

\begin{proof}
Combining Lemma~\ref{P-properties} (4), Lemma~\ref{short}, and Theorem~\ref{lo2} we have:
\[
p|C(h)| = |P(h)| = |\ker q||C(q(h))| = \left|\Z_{p}^{3}\right||C(q(h))| = p^{3}|C(q(h))|.
\]
And so, $|C(h)| = p^{2}|C(q(h))|.$
\end{proof}

\begin{Remark}
Since $Z(\h(p^n))$ is cyclic, we can assume that any generating set of a subgroup $H$ of $\h(p^n)$ contains at most one central element.
\end{Remark}

\begin{Corollary}\label{premain2}
Let $H \le \h(p^{n}).$

\noindent (1) If $HZ(\h(p^{n}))/Z(\h(p^{n}))$ is non-trivial and cyclic, then $|C(H)| = p^{2}|C(q(H))|.$

\noindent (2) If $H \leq Z(\h(p^{n})),$ then $|C(H)| = p^{3}|C(q(H))|.$
\end{Corollary}

\begin{proof}
(1) Let $H = \left< h_{1}, \, h_{2}  \right>$ with $h_{1} \in \h(p^{n})^{\dagger}$ and $h_{2} \in Z(\h(p^n)).$
By Lemma~\ref{premain1}, we obtain:
\begin{eqnarray*}
|C(H)| & = & |C(h_{1}) \cap  C(h_{2}) | = |C(h_{1}) \cap \h(p^{n})| = |C(h_{1})| = p^{2}|C(q(h_{1}))|\\
       & = & p^{2}\left|C(q(h_{1})) \cap \h \left(p^{n - 1}\right)\right| = p^{2}|C(q(h_{1})) \cap C(q(h_{2}))|\\
       & = & p^{2}|C(\left<q(h_{1}), \, q(h_{2})\right>)| = p^{2}|C(q(H))|.
\end{eqnarray*}

\noindent (2) If $H$ is central, then $C(H) = \h(p^n)$ and $C(q(H)) = \h \left(p^{n - 1}\right).$ Therefore, $|C(H)| = p^{3n} = p^{3}p^{3n - 3} = p^{3}|C(q(H))|.$
\end{proof}

By Lemma~\ref{PnoC2}, $P(H)/C(H)$ is trivial whenever $H \leq Z(\h(p^{n})).$ The next lemma describes $P(H)/C(H)$ when $H$ is not in $Z(\h(p^{n})).$

\begin{Lemma}\label{diagramChasing}
Let $H \le \hb{n}$ and $H \nsubseteq Z(\h(p^{n})).$ Suppose $H^{\dagger} = \{h_{1}, \, \ldots, \, h_{m}\}.$ The canonical homomorphism $P(H) \rightarrow \prod_{h \in H^{\dagger}} P(h)$ defined by $x \mapsto (x, \, \ldots, \, x)$ induces an isomorphism $\theta: P(H)/C(H) \rightarrow \mathbb{Z}^{k}_{p},$ where $k = 1$ or $2$.
\end{Lemma}

\begin{proof}
We have the following commutative diagram:
$$\xymatrix{
&1&1\\
&P(H)/C(H)\ar[u]&P(h_{1})/C(h_{1})\times\cdots\times P(h_{m})/C(h_{m})\ar[u]\\
1\ar[r]&P(H)\ar[u]\ar[r]&P(h_{1})\times \cdots\times P(h_{m})\ar[u]\\
1\ar[r]&C(H)\ar[u]\ar[r]&C(h_{1})\times \cdots \times C(h_{m})\ar[u]\\
&1\ar[u]&1\ar[u]
}$$
\noindent where the homomorphism from $P(H)$ to $P(h_{1}) \times \cdots \times P(h_{m})$ is the one in the statement of the lemma. The quotient groups in the diagram make sense because of Lemma~\ref{CnormalinP}. By Theorem \ref{lo2}, we have $P(h_{i})/C(h_{i})\cong \mathbb{Z}_{p}$ for each $h_{i} \in H^{\dagger}.$

We construct a group homomorphism $\hat{\theta} : P(H)/C(H) \rightarrow \mathbb{Z}_{p}^{m}$ as follows: let $\overline{x} \in P(H)/C(H)$ where $x \in P(H)$. Push right to $(x, \, \ldots, \, x)$ and then up to $\displaystyle \Pi_{i=1}^{m} P(h_i) / C(h_i) \cong \mathbb Z_{p}^{m}.$ Observe that if $a, \, b \in P(H)$ are such that $ab^{-1} \in C(H),$ then $ab^{-1} \in C(h_{i})$ for all $1 \le i \le m.$ Thus, $\hat{\theta}$ is a well-defined homomorphism and extends the previous commutative diagram in the following way:
$$\xymatrix{
&1&1\\
&P(H)/C(H)\ar[u]\ar@{-->}[r]^{\hat{\theta}}&\mathbb{Z}^{m}_{p}\ar[u]\\
1\ar[r]&P(H)\ar[u]\ar[r]&P(h_{1})\times\cdots\times P(h_{m})\ar[u]\\
1\ar[r]&C(H)\ar[u]\ar[r]&C(h_{1})\times \cdots \times C(h_{m})\ar[u]\\
&1\ar[u]&1\ar[u]
}$$

It is clear that $\hat{\theta}$ is a monomorphism; for if $\overline{x} \in \ker \hat{\theta} \leq \displaystyle P(H)/C(H)$
where $x \in P(H)$, then $x$ commutes with every $h_{i}$. And so, $\hat{\theta}$ induces an isomorphism $\theta : P(H)/C(H) \rightarrow \mathbb{Z}_{p}^{k}$ for some $k \in \left \{ 1, \, \ldots, \, m \right \}.$

We claim that $k = 1$ or $2.$ We know that
\[
P(H)/Z(\hb{n}) \le \hb{n}/Z(\hb{n}) \cong \mathbb{Z}_{p^{n}}^{2}.
\]
Thus, $P(H)/Z(\hb{n})$ has either one or two generators. Since $Z(\hb{n})\le C(H)$, there is a surjective map $P(H)/Z(\hb{n})\to P(H)/C(H).$ Thus, $k = 1$ or $2.$
\end{proof}

We will want to identify elements of $\h(p^{n})$ that are equal modulo its center. We introduce the following notation.

\begin{Notation} \label{n:EquivCenter}
Let $x, \, y \in \h(p^{n}).$ We will write $x \sim y$ whenever $xy^{-1} \in Z(\h(p^{n})).$
\end{Notation}

It is easy to see that $x \sim y$ if and only if $x|_{1} = y|_{1}$ and $x|_{3} = y|_{3}.$

\begin{Lemma} \label{l:EquivCentPseudo}
If $x, \, y \in \h(p^{n})$ and $x \sim y,$ then $[w, \, x] = [w, \, y]$ for any $w \in \h(p^{n}).$ In addition, $C(x) = C(y)$ and $P(x) = P(y).$
\end{Lemma}

\begin{proof}
It follows from Lemma~\ref{conmutador} (2).
\end{proof}

Lemma~\ref{pForm} gives a way of expressing an element $x \in P_{\ell}(h)$ in terms of an element of $C(h).$ A preferred form can be provided whenever the first or third component of $x$ is a unit in $\Z_{p^{n}}.$

\begin{Definition}
An element $x = (x_{1}, \, x_{2}, \, x_{3})$ in $\h(p^{n})$ is called {\it non-degenerate} if $p \nmid x_{1}$ or $p \nmid x_{3}.$ Otherwise, $x$ is {\it degenerate}.
\end{Definition}

Clearly, $x$ is non-degenerate if and only if either $x|_{1}$ or $x|_{3}$ is a unit in $\Z_{p^{n}}.$

\begin{Lemma}\label{nonDeg}
Suppose that $x, \, y \in \h(p^{n})$ and $x$ is non-degenerate. Then $y \in C(x)$ if and only if there exists a unique  $0 \le k < p^{n}$ such that $y \sim x^{k}.$
\end{Lemma}

\begin{proof}
If there exists $0 \le k < p^{n}$ such that $y \sim x^{k},$ then $[x, \, y] = \Bigl[x, \, x^{k}\Bigr] = 1$ by Lemma~\ref{l:EquivCentPseudo}. Thus, $y \in C(x).$

Conversely, suppose that $y \in C(x).$ By Corollary~\ref{C},
\begin{equation}\label{e:nonDeg}
x|_{1} y|_{3} \hbox{ mod } p^{n} = x|_{3}y|_{1} \hbox{ mod } p^{n}.
\end{equation}
Since $x$ is non-degenerate, either $p \nmid x|_{1}$ or $p \nmid x|_{3}.$ Let us assume that $p\nmid x|_{1}.$ Then $x|_{1}$ has an inverse, say $l,$ in the ring $\Z_{p^{n}}.$ Thus, (\ref{e:nonDeg}) can be rewritten as
\[
y|_{3} = l x|_{3}y|_{1} \hbox{ mod } p^{n}.
\]
Set $k = l y|_{1} \hbox{ mod } p^{n}.$ Using Lemma~\ref{known}, we get
\[
x^{k} = (x|_{1}, \, x|_{2}, \, x|_{3})^{k} = (k x|_{1}, \, \hat{x}, \, k x|_{3}) \sim (l x|_{1} y|_{1}, \, y|_{2}, \, l y|_{1} x|_{3}) = (y|_{1}, \, y|_{2}, \, y|_{3}),
\]
where $\hat{x} = \left( kx|_{2} + \displaystyle {k \choose 2}x|_{1} x|_{3}\right) \hbox{mod } p^{n}.$ A similar proof holds when $p \nmid x|_{3}.$

We now establish the uniqueness of $k$. Suppose that $y \sim x^{k_1} \sim x^{k_2}$ for some $0 \leq k_1 < p^{n}$ and $0 \leq k_2 < p^{n}.$ Assume, without loss of generality, that $p$ does not divide $x |_1.$ By Lemma~\ref{known}, $k_1 x |_1 = k_2 x_1.$ Since $x \vert_1$ is a unit in $\mathbb Z_{p^n}$, we conclude that $k_1 = k_2.$
\end{proof}

The next corollary gives an explicit form of the elements of $P_{\ell}(h)$ for a non-degenerate element $h$.

\begin{Corollary}\label{specificPForm}
Let $h = \Bigl(r_{1}p^{k_{1}}, \, r_{2}p^{k_{2}}, \, r_{3}p^{k_{3}}\Bigr) \in \h(p^{n})$ be non-degenerate. If $x \in P_{\ell}(h)$ for some $0 \leq \ell < p,$ then
\begin{eqnarray*}
x & = & \Bigl(wr_{1}p^{k_{1}}, \, \widehat{h}, \, wr_{3}p^{k_{3}} + r_{1}^{-1} \ell p^{n - k_{1} - 1}\Bigr) \quad \text{if} \quad r_{1}\ne 0 \hbox{ \ \ \ and}\\
x & = & \Bigl(wr_{1}p^{k_{1}} - r_{3}^{-1} \ell p^{n - k_{3} - 1}, \, \widehat{h}, \, wr_{3}p^{k_{3}}\Bigr) \quad \text{if} \quad r_{3}\ne 0,
\end{eqnarray*}
where $0 \le w < p^{n}$ and $0 \le \widehat{h} < p^{n}.$
\end{Corollary}

\begin{proof}
Let $h = \left(r_1p^{k_1}, \, r_{2}p^{k_2}, \, r_{3}p^{k_{3}} \right)$ be non-degenerate. Assume that $r_1 \neq 0$ and $0 \leq k_1 < n.$ Fix $0 \leq \ell <p$ and let $x \in P_{\ell}(h).$ By Lemma~\ref{pForm}, there exists $a \in C(h)$ such that
\[
x = \Bigl( a |_1, \, a |_2, \, a |_3 + r_1^{-1} \ell p^{n - k_1 - 1}  \Bigr).
\]
Lemma~\ref{nonDeg} allows us to choose $0 \leq w < p^n$ such that $a \sim h^{w} = \left( wr_1 p^{k_1}, \, \widehat{h}, \, wr_3p^{k_3} \right)$ for some $0 \leq \widehat{h} < p^n.$ Thus, we can set $x = \left( wr_1p^{k_1}, \, \widehat{h}, \, wr_3 p^{k_3} + r_{1}^{-1} \ell p^{n - k_1 - 1} \right).$
The case where $r_3 \neq 0$ is proven similarly.
\end{proof}

\begin{Definition}\label{d:EtaNu}
Let $h = \Bigl(r_{1}p^{k_{1}}, \, r_{2}p^{k_{2}}, \, r_{3}p^{k_{3}}\Bigr) \in \h(p^{n})^{\dagger}.$ If $r_{i} = 0,$ then we set $k_{i} = n.$ We define the following:
\begin{enumerate}
\item $\nu(h) = \min\{k_{1}, \, k_{3}\}$;
\item For $S = \{h_{1}, \, \ldots, \, h_{m}\} \subseteq \h(p^{n})^{\dagger},$ we set
\[
\nu(S) = \nu(h_{1}, \, \ldots, \, h_{m}) = \min\{\nu(h_{1}), \, \ldots, \, \nu(h_{m})\};
\]
\item $\n{h} = \Bigl(r_{1}p^{k_{1} - \nu(h)}, \, 0, \, r_{3}p^{k_{3} - \nu(h)} \Bigr) \in \h(p^{n}).$
\end{enumerate}
\end{Definition}

The next lemma follows from the definitions and Lemma \ref{conmutador} (1).

\begin{Lemma} \label{l:PropNu}
Suppose that $h = \Bigl(r_{1}p^{k_{1}}, \, r_{2}p^{k_{2}}, \, r_{3}p^{k_{3}}\Bigr) \in \h(p^{n})^{\dagger}.$ The following hold:
\begin{enumerate}
\item $0 \le \nu(h) < n$;
\item If $h$ is degenerate, then $\nu(h) > 0$;
\item If $h$ is non-degenerate, then $\nu(h) = 0$;
\item $\n{h} \in C(h)$.
\end{enumerate}
\end{Lemma}

Given a non-central element $h \in \h(p^{n})$ and an element $z \in P(h),$ we would like to find an element $y \in \h(p^{n})$ of a specific type such that $z \sim y.$ This is the point behind Lemma~ \ref{representation1}. First we provide some notation.

\begin{Notation}\label{n:Brace}
Let $h = \Bigl(r_{1}p^{k_{1}}, \, r_{2}p^{k_{2}}, \, r_{3}p^{k_{3}}\Bigr) \in \h(p^{n})^{\dagger}.$ If $a = (a_{1}, \, a_{2}, \, a_{3})$ is any element in $\h(p^{n}),$ then we set
\[
\left\{a_{i}\right\}_{h} = a_{i} \hbox{ mod } p^{n - \nu(h)} \hbox{ \ for \ } i = 1, \, 2, \, 3.
\]
(Here we are viewing the $a_{i}$ and $\left\{a_{i}\right\}_{h}$ as integers.)

Next we define an element $\left\{a \right\}_{h}$ in $\h \left(p^{n - \nu(h)} \right)$ as such:
\[
\left\{a \right\}_{h} = \Bigl(\left\{a_{1}\right \}_{h}, \, \left\{a_{2}\right\}_{h}, \, \left\{a_{3}\right \}_{h}\Bigr)
\]
\textbf{We emphasize that $a \in \h(p^{n})$ and $\left\{a \right\}_{h} \in \h \left(p^{n - \nu(h)} \right).$ It is clearly possible for $a_{i}$ to equal $\left\{a_{i}\right\}_{h}.$}
\end{Notation}

\begin{Lemma}\label{surjective}
Let $h = \Bigl(r_{1}p^{k_{1}}, \, r_{2}p^{k_{2}}, \, r_{3}p^{k_{3}}\Bigr) \in \hb{n}^{\dagger}.$ The mapping $a \mapsto \{ a \}_{h}$ induces an epimorphism
$\alpha: C(h) \rightarrow C\Bigl(\bigl \{ \n h \bigr \}_h \Bigr).$
\end{Lemma}

\begin{proof}
We first show that $\alpha$ is a homomorphism. If $x, \, y \in C(h),$ then it is clear that $\left\{ xy \right\}_{h} = \left\{x \right\}_{h} \left\{y \right\}_{h}.$ We claim that if $z = (z_{1}, \, z_{2}, \, z_{3}) \in C(h),$ then $\left\{z\right\}_{h} \in C\Bigl(\bigl\{ \n h \bigr\}_{h} \Bigr).$ Indeed,
\[
\bigl[\left\{ \n h \right\}_{h}, \, \{z\}_{h}\bigr] \big|_{2} =
\left\{z_{3}r_{1}p^{k_{1} - \nu(h)} - z_{1}r_{3}p^{k_{3} - \nu(h)}\right\}_{h} \in \Z_{p^{n - \nu(h)}}.
\]
By Corollary~\ref{C}, $z \in C(h)$ implies that $[h, \, z] \vert_2 = 0.$ Thus,
\[
\left(z_{3}r_{1}p^{k_{1} - \nu(h)} - z_{1}r_{3}p^{k_{3} - \nu(h)}\right) \hbox{ mod } p^{n - \nu(h)} = 0
\]
or, equivalently, $\bigl[\left\{ \n h \right\}_{h} , \, \{z\}_{h}\bigr] \big|_{2} = 0.$ And so, $\{z\}_{h} \in C\Bigl(\bigl \{ \n h \bigr \}_h \Bigr)$ as claimed. Thus, $\alpha$ is a homomorphism.

Next we prove that $\alpha$ is onto. Let $w = (w_{1}, \, w_{2}, \, w_{3}) \in C\Bigl(\bigl\{ \n h \bigr\}_{h} \Bigr)$.
If we let $a = (a_{1}, \, a_{2}, \, a_{3}) \in \hb{n},$ where $a_{i} \hbox{ mod } p^{n - \nu(h)} = w_{i}$ for each $i = 1, \, 2, \, 3,$ then $\left\{a \right\}_h = w.$ Note that the $a_{i}$ exist (simply take $a_{i} = w_{i}$). We need to show that $a \in C(h).$ Since $w \in C\Bigl( \bigl\{ \n h \bigr\}_{h} \Bigr),$ we have $\bigl[\left\{ \n h  \right\}_h, \, w \bigr] \bigl |_{2} = 0$ by Corollary~\ref{C}. Hence,
\[
\biggl(w_{3}\left\{r_{1}p^{k_{1} - \nu(h)}\right\}_{h} - w_{1}\left\{r_{3}p^{k_{3}-\nu(h)}\right\}_{h}\biggr) \hbox{ mod }p^{n - \nu(h)} = 0
\]
by Lemma~\ref{conmutador}. Now, $\left\{a \right\}_h = w$ gives
\[
\biggl(\left\{a_{3}r_{1}p^{k_{1} - \nu(h)}\right\}_{h} - \left\{a_{1}r_{3}p^{k_{3}-\nu(h)}\right\}_{h}\biggr) \hbox{ mod }p^{n - \nu(h)} =0,
\]
which is the same as $\Bigl(\left\{a_{3}r_{1}\right\}_{h}p^{k_{1} - \nu(h)} - \left\{a_{1}r_{3}\right\}_{h}p^{k_{3}-\nu(h)}\Bigr) \hbox{ mod }p^{n - \nu(h)} = 0.$ Hence, $\Bigl(\left\{a_{3}r_{1}\right\}_{h}p^{k_{1}} - \left\{a_{1}r_{3}\right\}_{h}p^{k_{3}}\Bigr) \hbox{ mod }p^{n} = 0.$
By the Division Algorithm, there exist $s, \, t \in \N$ such that $a_{1}r_{3} = sp^{n - \nu(h)} + \{a_{1}r_{3}\}_{h}$ and $a_{3}r_{1} = tp^{n - \nu(h)} + \{a_{3}r_{1}\}_{h}.$ Thus, $\Bigl(a_{3}r_{1}p^{k_{1}} - a_{1}r_{3}p^{k_{3}}\Bigr)\hbox{ mod }p^{n} = 0.$ By Corollary~\ref{C}, $a \in C(h).$
\end{proof}

\begin{Lemma} \label{l:BetaEquiv}
Let $\alpha$ be the mapping defined in Lemma~\ref{surjective}. If $\alpha(x) \sim \alpha(y)$ in $\hb{n-\nu(h)},$ then there is an element in $C(h) \cap \ker \alpha$ of the form $\left(r^{\prime}_1, \, 0 , \, r^{\prime}_3 \right)^{p^{n-\nu(h)}}$ where $0 \leq r^{\prime}_1, \, r^{\prime}_3 < p^{\nu(h)}$ and $x \sim y \left(r^{\prime}_1, \, 0 , \, r^{\prime}_3 \right)^{p^{n-\nu(h)}} $ in $\hb{n}.$
\end{Lemma}

\begin{proof}
Let $x = (x_1, \, x_2, \, x_3)$ and $y = (y_1, \, y_2, \, y_3)$ be elements of $C(h)$ such that $\alpha(x) \sim \alpha(y).$ Then $\left\{x \right \}_h \sim \left \{y \right \}_h$ and, thus, $\{x_i \}_h = \{ y_i \}_h$ for $i = 1, \, 3.$ Hence, $x_i\hbox{ mod }p^{n - \nu(h)} = y_i\hbox{ mod }p^{n - \nu(h)}$ and, consequently, $x_i = y_i + l_i p^{n - \nu(h)}$ for some $l_i \in \mathbb Z$ whenever $i = 1, \, 3.$ Hence,
\[
x \sim \left(y_1 + l_1 p^{n-\nu(h)}, \, y_2, \, y_3 +  l_3 p^{n-\nu(h)} \right) \sim (y_1, \, y_2, \, y_3) \left( l_1, \, 0 , \, l_3 \right)^{p^{n-\nu(h)}}.
\]
Now, there exists $0 \leq r^{\prime}_i < p^{\nu(h)}$ for $i = 1, \, 3$ such that
\[
\left( l_1, \, 0 , \, l_3 \right)^{p^{n-\nu(h)}} \sim \left(r^{\prime}_1, \, 0 , \, r^{\prime}_3 \right)^{p^{n-\nu(h)}}.
\]
It readily follows from Lemma~\ref{l:EquivCentPseudo}
that $(r^{\prime}_1, \, 0,  \, r^{\prime}_3)^{p^{n - \nu(h)}} \in C(h) \cap \ker \alpha.$
\end{proof}

We now invoke the last two lemmas to prove the following useful result.

\begin{Lemma}\label{representation1}
Let $h = \Bigl(r_{1}p^{k_{1}}, \, r_{2}p^{k_{2}}, \, r_{3}p^{k_{3}}\Bigr)$ and $z \in \hb{n}^{\dagger}.$ Then $z \in P(h)$ if and only if
\begin{equation}\label{p1}
z \sim \n{h}^{w_1}\Bigl(0, \, 0, \, r_{1}^{-1}\ell_1 p^{n- 1 - k_{1}}\Bigr)\bigl(r^{\prime}_{11}, \, 0, \, r^{\prime}_{13}\bigr)^{p^{n - \nu(h)}} \qquad \text{if}\;r_{1} \ne 0
\end{equation}
and
\begin{equation}\label{p2}
z \sim \n{h}^{w_2}\Bigl(-r_{3}^{-1} \ell_2 p^{n- 1- k_{3} }, \, 0, \, 0\Bigr)\bigl(r^{\prime}_{21}, \, 0, \, r^{\prime}_{23}\bigr)^{p^{n - \nu(h)}} \qquad \text{if}\;r_{3}\ne 0
\end{equation}
for some integers $w_{i}$, $\, 0 \le {\ell}_{i} < p,$ and $0 \le r_{i1}^{\prime}, \, r_{i3}^{\prime} < p^{\nu(h)}$\, ($i=1, \, 2$).
\end{Lemma}

Note that (\ref{p1}) and (\ref{p2}) in the statement of Lemma~\ref{representation1} are not mutually exclusive. When $r_1$ and $r_3$ are both non-zero, either can be used.

\begin{proof}
Suppose that $r_{1} \ne 0$ and $z \in P(h).$ We assert that $z$ satisfies (\ref{p1}). By Lemma~\ref{pForm}, we have
\[
z \sim \Bigl(a_{1}, \, a_{2}, \, a_{3} + r_{1}^{-1}\ell_{1} p^{n - 1 - k_{1}}\Bigr) \sim(a_{1}, \, a_{2}, \, a_{3})\Bigl(0, \, 0, \, r_{1}^{-1}\ell_{1} p^{n - 1 - k_{1}}\Bigr)
\]
for some $a = (a_{1}, \, a_{2}, \, a_{3}) \in C(h)$ and $0 \leq \ell_{1} < p.$ By Lemma~\ref{surjective}, $\alpha(a) = \left\{a\right\}_{h}$ in $C\Bigl(\left \{ \n h \right \}_h\Bigr).$ We check that $\left \{ \n h \right \}_h$ is non-degenerate. Since $h$ is non-central, we assume first that $r_1 \neq 0$ and
$r_{3} = 0.$ In this case, we get $k_3 = n$ and, thus, $\nu(h) = k_1.$ Hence, $\left \{ \n h \right\}_{h} = \Bigl(\left \{ r_1 \right \}_h, \, 0 , \, \left \{r_3 \right\}_h p^{k_3 - k_1} \Bigr).$ And so, $\left \{ \n h \right \}_h$ is non-degenerate. The proof is similarly straightforward whenever $r_i \neq 0$ for $i = 1, \, 3$; or $r_1 = 0$ and $r_3 \neq 0$.

By Corollary~\ref{specificPForm}, recalling that $P_{0}\Bigl(\left \{ \n h \right \}_h\Bigr) = C\Bigl(\left \{ \n h \right \}_h\Bigr),$ we obtain
\[
\left\{a \right\}_{h} \sim \left(w \{r_{1}\}_h p^{k_{1} - \nu(h)}, \, 0, \, w\{r_{3}\}_h p^{k_{3} - \nu(h)}\right) \sim \{\n{h} \}_h^{w}
\]
for some $0 \le w < p^{n - \nu(h)}$. By Lemma \ref{l:BetaEquiv}, it follows that $a\sim \n{h}^{w}\bigl(r^{\prime}_{1}, \, 0, \, r^{\prime}_{3}\bigr)^{p^{n - \nu(h)}}$ with $0 \le r_{1}^{\prime}, \, r_{3}^{\prime} < p^{\nu(h)}.$ Thus,
\begin{align*}
z & \sim a\Bigl(0, \, 0, \, r_{1}^{-1}\ell_{1} p^{n - 1 - k_{1}}\Bigr) \sim \n{h}^{w}\bigl(r^{\prime}_{1}, \, 0, \, r^{\prime}_{3}\bigr)^{p^{n - \nu(h)}}\Bigl(0, \, 0, \, r_{1}^{-1}\ell_{1} p^{n - 1 - k_{1}}\Bigr)\\
  & \sim \n{h}^{w}\Bigl(0, \, 0, \, r_{1}^{-1}\ell_{1} p^{n - 1 - k_{1}}\Bigr)\bigl(r^{\prime}_{1}, \, 0, \, r^{\prime}_{3}\bigr)^{p^{n - \nu(h)}}.
\end{align*}
This proves the assertion. A similar argument shows that if $r_{3} \ne 0$ and $z \in P(h),$ then $z$ satisfies (\ref{p2}).

Conversely, suppose that $z$ satisfies (\ref{p1}). It follows from Lemmas~\ref{producttosum} and \ref{l:EquivCentPseudo} that $[h, \, z]|_{2} = \ell_{1} p^{n - 1}.$ Thus, $z \in P(h)$. We proceed similarly if $r_{3} \ne 0$.
\end{proof}

\begin{Remark} \label{r:Centralizer}
We observe that $z \in C(h)$ if and only if $\ell_{1} = 0$ ($\ell_{2} = 0$) in Lemma~\ref{representation1}. For if $z$ commutes with $h$, then the last calculation of the proof of Lemma~\ref{representation1} gives that $\ell_{1} = 0$ ($\ell_{2} = 0$). For the converse we set $\ell_{1} = 0$ ($\ell_{2} = 0$) in (\ref{p1}) ((\ref{p2}), respectively) of the said lemma.
\end{Remark}

Our next goal is to describe certain sufficient conditions on a subgroup $H$ of $\h(p^n)$ so that $P(H)/C(H)\cong \mathbb{Z}_p^2.$ Some preparation is needed.

\begin{Definition}
A generating set $S$ for $H \le \hb{n}$ is called \emph{special} if either
\begin{enumerate}
\item $S^{\dagger} = \emptyset$; or
\item there exists $h_1 \in S^{\dagger}$ such that $\nu(S^\dagger)=\nu(h_1)$ and $h_2 \not \sim h_1^{w}$ for all $h_2 \in S^{\dagger} - \{h_1\}$ and any non-negative integer $w$.
\end{enumerate}
\end{Definition}

\begin{Lemma}
Every subgroup $H$ of $\hb{n}$ has a special generating set.
\end{Lemma}

\begin{proof}
Let $S$ be a generating set for $H$ containing at least one non-central element. Choose $h_1 \in S^{\dagger}$ such that $\nu \left( S^{\dagger} \right) = \nu(h_1)$. If $S^{\dagger} = \{h_1\},$ then we are done.

Suppose that there exists $h_2 \in S^{\dagger} - \{h_1\}$ satisfying $h_2 \sim h_1^{w_1}$ for some $0 < w_{1};$ that is, $h_1^{w_1}h_{2}^{-1} \in Z(\hb{n}).$ Put $h_1^{w_1}h_{2}^{-1} = c_{1}$ and let $S_{1} = \left(S-\{h_{2}\}\right)\cup \{c_{1}\}.$ Clearly, $S_{1}$ is also a generating set for $H$. If $h_3 \not \sim h_1^{w_2}$ for any $h_3 \in S_{1}^{\dagger} - \{h_1\}$ and any $w_2 \geq 0,$ then $S_{1}$ is special. Otherwise, we repeat the same procedure starting with $S_{1}$ and obtain another generating set for $H.$ Continuing in this way will eventually produce a special generating set for $H$.
\end{proof}

\begin{Definition}\label{d:Supercommuting}
Let $h_{1} \in \hb{n}^{\dagger}$ and $h_{2} \in \hb{n}.$ We say that $h_2$ is \emph{supercommuting} with $h_{1}$ if $h_2 \sim \bigl(r_1^{\prime}, \, 0, \, r_3^{\prime}\bigr)^{p^{n - \nu(h_1)}}$ with $0 \le r_1^{\prime}, \, r_3^{\prime} < p^{\nu(h_1)}$ and $r_1^{\prime} + r_3^{\prime} > 0.$
\end{Definition}

\begin{Remark}\label{r:Supercommuting}
The condition $r_1^{\prime} + r_3^{\prime} > 0$ implies that $h_{2} \in \hb{n}^{\dagger}.$ In addition, it can be shown that $h_{2}$ and $h_{1}$ commute whenever $h_{2}$ is supercommuting with $h_{1}.$
\end{Remark}

\begin{Lemma}\label{propCommute}
Let $S$ be a special generating set for $H \leq \hb{n}$ with $S^{\dagger} \neq \emptyset,$ and let $h \in S^{\dagger}$ such that $\nu(h) = \nu \left(S^{\dagger} \right).$ If $h$ is non-degenerate, then $S^{\dagger}$ contains no elements that are supercommuting with $h$.
\end{Lemma}

\begin{proof}
If $h$ is non-degenerate, then $\nu(h) = 0$ and, thus, $p^{n - \nu(h)} = p^{n}$. Assume that $\tilde{h} \in S^{\dagger}$ is supercommuting with $h,$ where $\tilde{h} \neq h.$ By definition, there exist $0 \le r_1^{\prime}, \, r_3^{\prime} < 1$ with $r_1^{\prime} + r_3^{\prime} > 0$ such that $\tilde{h} \sim \Bigl(r_1^{\prime}, \, 0, \, r_3^{\prime}\Bigr)^{p^{n}}.$ This is clearly impossible.
\end{proof}

The next notation is a slight repeat of Notation~\ref{n:Factored_h}.

\begin{Notation}\label{Nh}
Let $h_1,\ldots,h_m\in \hb{n}^\dagger$. We write $h_{j} = \bigl(r_{j1}p^{k_{j1}}, \, h_{j2}, \, r_{j3}p^{k_{j3}}\bigr)$, where either

\noindent (1) \ $r_{ji} = 0,$ in which case we set $k_{ji} = n,$ or

\vspace{.05in}

\noindent (2) \ $0 \leq k_{ji} < n,$ $0 < r_{ji} < p^{n},$ and $p \nmid r_{ji}$

\vspace{.05in}

\noindent for $i = 1, \, 3.$
\end{Notation}

\begin{Lemma}\label{super}
Let $S$ be a special generating set for $H \le \hb{n}$ with $S^{\dagger} \neq \emptyset$ such that $\nu \left(S^{\dagger} \right) = \nu(h)$ for some $h \in S^{\dagger}.$
\begin{enumerate}
\item If $h_1 \in S^{\dagger} - \{h\}$ commutes with $h,$ then
\[
h_1 \sim h^{w}\Bigl(r_1^{\prime}, \, 0, \, r_3^{\prime}\Bigr)^{p^{n - \nu(h)}}
\]
for some $0 \leq w$ and $0 \le r_1^{\prime}, \, r_3^{\prime} < p^{\nu(h_1)}$ with $r_1^{\prime} + r_3^{\prime} > 0.$ Moreover, if $r_{i}^{\prime} = r_{i}^{\prime\prime}p^{s_{i}} \neq 0$ for some $i = 1, \, 3$ with $p \nmid r_{i}^{\prime\prime}$ and $s_{i} \geq 0,$ then $2 \nu(h) \le n + s_{i}$.
\item There exists a special generating set $T$ for $H$ which satisfies the following properties:
\begin{enumerate}
\item $\nu(h) = \nu \left(T^{\dagger} \right)$;
\item All elements of $T^{\dagger}-\{h\}$ that commute with $h$ are supercommuting with $h$.
\end{enumerate}
\end{enumerate}
\end{Lemma}

\begin{proof}
For both parts of the proof we let $h = \bigl(r_{1}p^{k_{1}}, \, h_{2}, \, r_{3}p^{k_{3}}\bigr)$ and assume that $\nu(h) = k_{1}$ (the other case is handled similarly).

By Lemma~\ref{representation1} and Remark~\ref{r:Centralizer}, we have
\[
h_1 \sim {}_{\nu} h^{w_1}\Bigl(r_1^{\prime}, \, 0, \, r_3^{\prime}\Bigr)^{p^{n - k_{1}}} \sim \Bigl(r_{1}w_1 + r_1^{\prime}p^{n - k_{1}}, \, 0, \,r_{3}w_{1}p^{{k_{3} - k_{1}}} + r_3^{\prime}p^{n - k_{1}}\Bigr)
\]
with $0 \le w_1 <p^{n - k_{1}}$ and $0 \le r_1^{\prime}, \, r_3^{\prime} < p^{{k_{1}}}.$ It follows from the comment preceding Lemma~\ref{l:EquivCentPseudo} that $p^{{k_{1}}} \mid w_1$ and $p^{k_{1}}\mid r^{\prime}_ip^{n - k_{1}}$ for $i = 1, \, 3.$ This means that $w_1 = wp^{{k_{1}}}$ for some $w \geq 0$. Hence,
\[
{}_{\nu} h^{w_1} = {}_{\nu}h^{wp^{k_{1}}} = \left({}_{\nu}h^{p^{k_{1}}}\right)^{w} \sim h^w.
\]

Now, if $r^{\prime}_1 = r^{\prime}_3 = 0,$ then $h_1 \sim {}_{\nu} h^{w_1} \sim h^w.$ This contradicts the fact that $S$ is special, and hence, $r^{\prime}_1 + r^{\prime}_3 > 0.$ Since $p^{k_{1}}\mid r^{\prime}_ip^{n - k_{1}},$ we have $p^{k_{1}}\mid r^{\prime\prime}_ip^{n - k_{1} + s_{i}}.$ Thus, $k_{1} \le n - k_{1} + s_{i};$ that is, $2k_{1} \le n + s_{i}.$ This completes the proof of (1).

To prove (2), suppose there exists $h_{1} \in S^{\dagger}$ with $h_{1} \neq h$ that commutes with $h$ but is not supercommuting with $h$. Write $h_1\sim h^{w}\Bigl(r_1^{\prime}, \, 0, \, r_3^{\prime}\Bigr)^{p^{n - k_1}}$ as in (1) with $0 < w < o(h).$
Since
\[
h^{o(h) - w}h_1 \sim h^{o(h) - w} h^w\Bigl(r_1^{\prime}, \, 0, \, r_3^{\prime}\Bigr)^{p^{n -k_1}} \sim \Bigl(r_1^{\prime}, \, 0, \, r_3^{\prime}\Bigr)^{p^{n - k_1}},
\]
there exists a central element $z_1$ such that
\[
h^{o(h) - w} h_1 = \Bigl(r_1^{\prime}, \, 0, \, r_3^{\prime}\Bigr)^{p^{n - k_1}} z_1.
\]
A new generating set $T_1$ for $H$ is obtained by replacing $h_1$ with $\Bigl(r_1^{\prime}, \, 0, \, r_3^{\prime}\Bigr)^{p^{n - k_1}}$ and $z_1$. Writing $r_1^{\prime}$ and $r_3^{\prime}$ as in (1), and using the fact that $2k_1 \leq n + s_i$  for $i = 1, \, 3$, gives that $\nu(h) = \nu \left(T_{1}^{\dagger} \right)$. If all elements of $T_{1}^{\dagger} - \{h\}$ that commute with $h$ are supercommuting with $h,$ then $T_1$ is a special generating set and we set $T = T_1.$ Otherwise, we repeat the same procedure until a special generating set $T$ with the desired properties is obtained.
\end{proof}

\begin{Remark} \label{r:Super}
By Lemma~\ref{super}, we can assume that if $S$ is any special generating set for $H$ with $S^{\dagger} \neq \emptyset$ such that $\nu({h}) = \nu \left(S^{\dagger} \right)$ for some $h \in S^{\dagger}$ and $h_{1} \in S^{\dagger} - \{h\}$ commutes with $h,$ then $h_{1}$ is supercommuting with $h.$
\end{Remark}

\begin{Notation} \label{n:Phi}
For any non-negative integer $x = rp^{k},$ where $r$ and $k$ are integers and $p \nmid r,$ we let $\varphi(x) = k$. If $r = 0,$ then we write $rp^{k}$ as $p^{n}$. Thus, if $x = 0$ then we have $\varphi(x) = n.$
\end{Notation}

We set up some data for the next definition. Let $H \leq \hb{n}$ and $S$ a special generating set for $H$ with $S^{\dagger} \neq \emptyset.$ Suppose that $h_{1}, \, h_{2} \in S^{\dagger}$ are distinct commuting elements and $\nu\left(S^\dagger \right) = \nu(h_{1}).$ By Remark~\ref{r:Super}, $h_{2}$ is supercommuting with $h_{1}$, and thus, $h_{2} \sim \bigl(r_{1}^{\prime}, \, 0, \, r_{3}^{\prime}\bigr)^{p^{n - \nu(h_{1})}}$ for some $0 \leq r_{1}^{\prime}, \, r_{3}^{\prime} < \nu(h_{1})$ with $r_{1}^{\prime} + r_{3}^{\prime} > 0.$

\begin{Definition} \label{d:Proper}
We say that $h_{2}$ \emph{commutes properly} with $h_{1}$ if
\[
\varphi\Bigl(\Bigl[\Bigl(r^{\prime}_1, \, 0, \, r_3^{\prime}\Bigr), \, {}_{\nu}h_1\Bigr] \Big |_2\Bigr) < \nu(h_{1}).
\]
Otherwise, $h_2$ \emph{commutes improperly} with $h_{1}.$
\end{Definition}

\begin{Lemma} \label{l:Properly}
Let $S$ be a special generating set for $H \leq \hb{n}$ with $S^{\dagger} \neq \emptyset.$ Suppose that $h_1 \in S^{\dagger}$ is such that $\nu \left(S^{\dagger} \right) = \nu(h_1).$ Then $S^{\dagger}$ contains no elements that commute improperly with $h_1$.
\end{Lemma}

\begin{proof}
Assume, on the contrary, that there exists $h_2 \in S^{\dagger} - \{h_1\}$ that commutes improperly with $h_1$. Using the data before Definition~\ref{d:Proper} we note that $h_2 \sim \Bigl(r_1^{\prime}, \, 0, \, r_3^{\prime}\Bigr)^{p^{n - \nu(h_1)}}$ with $0 \le r_1^{\prime}, \, r_3^{\prime} < p^{\nu(h_1)}$ and $r_1^{\prime} + r_3^{\prime} > 0.$ Lemma~\ref{producttosum} (3) (with $x = z$ and $y = w$), together with a straightforward induction, gives
\begin{align} \label{e:SuperCommuting}
\Bigl[\Bigl(r_1^{\prime}, \, 0, \, r_3^{\prime}\Bigr)^{p^{n - \nu(h_1)}}, \, {}_{\nu}h_1\Bigr] \Big |_2 & = p^{n - \nu(h_1)}\Bigl[\Bigl(r_1^{\prime}, \, 0, \, r_3^{\prime}\Bigr), \, {}_{\nu}h_1 \Bigr] \Big |_2.
\end{align}
We are assuming that $h_{2}$ commutes improperly with $h_{1}$, so that (\ref{e:SuperCommuting}) is divisible by $p^n$. Hence, $h_2$ commutes with ${}_{\nu}h_1$.
Since ${}_{\nu}h_1$ is non-degenerate, $h_{2} \sim {}_{\nu}h_{1}^{w}$ by Lemma~\ref{nonDeg}, where $0 < w < p^n.$ Thus,
\[
h_2 \sim \left(w r_{11}p^{k_{11} - \nu(h_1)}, \, 0 , \, wr_{13}p^{k_{13} - \nu(h_1)} \right).
\]
Now, $\nu(h_1) \le \nu(h_2)$ because $h_2 \in S^{\dagger}$ and $\nu(S^{\dagger}) = \nu(h_1).$ This means that the exponents of $p$ in $\left(h_2 \right) |_{1}$ and $\left( h_2 \right) |_{3}$ have to be at least $\nu(h_{1}).$ And so, we must have $w = ap^{\nu(h_1) + k}$ for some $k \ge 0$ and $p\nmid a$. Hence,
\[
h_2 \sim {}_{\nu}h_1^w \sim {}_{\nu}h_1^{ap^{\nu(h_1) + k}} \sim \left({}_{\nu}h_1^{p^{\nu(h_1)}}\right)^{ap^k} \sim h_1^{ap^{k}}.
\]
This gives a contradiction because $S$ is special.
\end{proof}

\begin{Definition}
For distinct elements $h_{s}, \, h_{t} \in \hb{n},$ we define
\[
\mu(h_{s}, \, h_{t}) = \min\{k_{s3} + k_{t1}, \, k_{s1} + k_{t3}\} - \nu(h_{s}, \, h_{t}).
\]
\end{Definition}

\begin{Lemma}\label{l:MU}
If $\mu(h_{1}, \, h_{2}) \ge n$ for some $h_{1}, \, h_{2} \in \hb{n},$ then $h_{1}$ and $h_{2}$ commute.
\end{Lemma}

\begin{proof}
Set $\lambda(h_1, \, h_2) = \min\{k_{11} + k_{23}, \, k_{21}+k_{13}\}$. Since $n \le \mu(h_1, \, h_2) \le \lambda(h_1, \, h_2),$ we have
\begin{align*}
[h_1, \, h_2]|_2 & = r_{11}r_{23}p^{k_{11} + k_{23}} - r_{21}r_{13}p^{k_{21} + k_{13}}\\
                 & = p^{\lambda(h_1, \, h_2)}\Bigl(r_{11}r_{23}p^{k_{11} + k_{23} - \lambda(h_1, \, h_2)} - r_{21}r_{13}p^{k_{21} + k_{13} - \lambda(h_1, \, h_2)}\Bigr) = 0.
\end{align*}
Hence, $h_{1}$ and $h_{2}$ commute.
\end{proof}

As mentioned earlier we are now prepared to prove that (under certain conditions on $H\le \h(p^{n}), P(H)/C(H)$ is isomorphic to $\mathbb{Z}_{p^2}$. This is one of our major theorems.

\begin{Theorem}\label{main1}
Let $H \le \hb{n}$ and $S$ be a special generating set for $H$ such that $S^{\dagger} = \{h_{1}, \, h_{2}\}.$ If $h_{1}$ and $h_{2}$ do not commute or $h_{2}$ commutes properly with $h_{1},$ then
\[
P(H)/C(H)\cong \mathbb{Z}_p^2.
\]
\end{Theorem}

\begin{proof}
We construct elements $z_{1}, \, z_{2}$ which satisfy the following two conditions:
\begin{enumerate}
\item $z_{1} \in P(h_{1}) - C(h_{1})$ and $z_{1} \in C(h_{2})$;
\item $z_{2} \in P(h_{2}) - C(h_{2})$ and $z_{2} \in C(h_{1})$.
\end{enumerate}
Since $C(h_{1}) \subseteq P(h_{1})$ and $C(h_{2}) \subseteq P(h_{2}),$ Lemmas~\ref{P-properties} (5, 9) and \ref{l:P(H)=P(SN)} give
\[
z_{1}, \, z_{2} \in P(h_{1}) \cap P(h_{2}) = P(\{ h_1, \, h_2 \}) = P\left(S^{\dagger}\right) = P(H).
\]
The theorem will follow from Lemma~\ref{diagramChasing} by showing that $k = 2$ in the said lemma. For if $P(H)/C(H) \cong \mathbb Z_p$, then $z_{1}^{\ell}z_{2}^{-1} \in C(H)$ for some integer $\ell$. Thus, in particular, $z_{1}^{\ell}z_{2}^{-1}h_{2} = h_{2}z_{1}^{\ell}z_{2}^{-1} = z_{1}^{\ell}h_{2} z_{2}^{-1}$ since $z_{1} \in C(h_{2})$. But this gives $z_{2} \in C(h_{2})$, a contradiction. The rest of the proof consists on the construction of $z_{1}$ and $z_{2}$.

We adopt the notation defined in (\ref{Nh}). To simplify notation, we let $\mu(2) = \mu(h_1, \, h_2)$. Throughout the proof we assume that $\nu(h_{1}, \, h_{2}) = \nu(h_{1}) = k_{11}$, and thus, $r_{11} \neq 0$. The case where $\nu(h_{1}) = k_{13}$ will follow similarly.
In light of Lemma~\ref{l:Properly}, and given that $S$ is special, there are two cases to consider: when $h_{1}$ and $h_{2}$ do not commute at all, and when $h_{2}$ commutes properly with $h_{1}.$ We handle, in turn, each of these cases.

\vspace{.15in}

\noindent \textbf{\underline{CASE I}: $h_1$ and $h_2$ do not commute}

Put
\[
r_{2} = r_{13}r_{21}p^{k_{13} + k_{21} - k_{11} - \mu(2)} - r_{11}r_{23}p^{k_{23} - \mu(2)}.
\]
Notice that $r_{2}$ is a unit in $\mathbb Z_{p^n}$ by definition of $\mu(2)$. For $i = 1, \, 2$, let
\begin{equation} \label{zi}
z_{i} = {}_{\nu}h^{w_{i}}_{1}\left(0, \, 0, \, (2 - i)r_{11}^{-1} p^{n - k_{11} - 1}\right), \hbox{ \ \ where }
\end{equation}
\begin{equation}\label{w0}
w_{i} = r_{2}^{-1} \bigl[(i - 1) - (2 - i)r_{11}^{-1}r_{21}p^{k_{21} - k_{11}}\bigr]p^{n - 1-\mu(2)}.
\end{equation}
Since $h_1$ and $h_2$ do not commute, we have $\mu(2) \le n - 1$ by Lemma~\ref{l:MU}. This, together with the fact that $r_{2}$ is a unit, gives that the $z_{i}$ are well-defined. Lemma~\ref{representation1} and Remark~\ref{r:Centralizer} immediately give that $z_{1} \in P(h_1) - C(h_{1})$ and $z_2 \in C(h_1)$.

We claim that $z_{2} \in P(h_{2}) - C(h_{2})$ and $z_{1} \in C(h_{2}).$ Invoking Lemma~\ref{producttosum}, we have:

\noindent $\bigl[h_{2}, \, z_{i}\bigr] \bigl |_{2} = \bigl[h_{2}, \, {}_{\nu}h^{w_{i}}_{1}\bigl(0, \, 0, \, (2 - i)r^{-1}_{11}p^{n - k_{11} - 1}\bigr)\bigr]\bigl |_{2}$

\vspace{.05in}

\noindent $ \ \ \ \ \ \ \ = \bigl[h_{2}, \, {}_{\nu}h^{w_{i}}_{1}\bigr] \bigl |_{2} + \bigl[h_{2}, \, \bigl(0, \, 0, \, (2 - i)r^{-1}_{11}p^{n - k_{11} - 1}\bigr)\bigr] \bigl |_{2}$

\vspace{.05in}

\noindent $ \ \ \ \ \ \ \ = \bigl(w_{i}\bigl(r_{21}r_{13}p^{k_{13} + k_{21} - k_{11}} - r_{11}r_{23}p^{k_{23}}\bigr) + (2 - i)r_{11}^{-1}r_{21}p^{n + k_{21} - k_{11} - 1}\bigr) \hbox{ mod } p^{n}$

\vspace{.05in}

\noindent $ \ \ \ \ \ \ \ = \bigl(w_{i}r_{2}p^{\mu(2)} + (2 - i)r_{11}^{-1}r_{21}p^{n + k_{21} - k_{11} - 1}\bigr) \hbox{ mod } p^{n}.$

\vspace{.05in}

\noindent Substituting the value of $w_i$ from (\ref{w0}) and simplifying the last expression gives $\left[ h_{2}, \, z_{i} \right] = (i - 1)p^{n - 1}.$ This proves the claim, thereby completing \textbf{\underline{CASE I}}.

\vspace{.15in}

\noindent \textbf{\underline{CASE II}: $h_2$ commutes properly with $h_1$}

Since $h_1$ and $h_2$ are commuting elements of $S^{\dagger}$ and $\nu(S^{\dagger}) = \nu(h_1) = k_{11}$, we remind the reader that $h_2$ is supercommuting with $h_1$ by Remark \ref{r:Super}. Thus, $h_2 \sim \Bigl(r_{21}^{\prime}, \, 0, \, r_{23}^{\prime}\Bigr)^{p^{n - k_{11}}}$ for some $0 \le r_{21}^{\prime}, \, r_{23}^{\prime} < k_{11}$ with $r_{21}^{\prime} + r_{23}^{\prime} > 0.$ We will show that $\mu(2) \leq n - 1$, as we will need this to define the $z_{i}$. We consider three distinct situations.
\begin{itemize}
\item For $j = 1, \, 3$, assume first that $r_{2j}^{\prime} \neq  0$, and write $r_{2j}^{\prime} = r_{2j}^{\prime \prime}p^{k^{(2)}_{j}} \ne 0$ where $k^{(2)}_{j} \geq 0$ and $p \nmid r_{2j}^{\prime \prime}.$ Then $2k_{11} \le n + k^{(2)}_{1}$ by Lemma~\ref{super}, and thus, $0 \le n - 2k_{11} + k^{(2)}_{1}.$ By Lemma~\ref{super}, $p^{k^{(2)}_3} \leq r_{23}^{\prime}<p^{k_{11}}$. This implies that $n + k^{(2)}_3 - k_{11} \leq n - 1$. By the statement before Lemma~\ref{l:EquivCentPseudo}, we have $\mu(2) = \mu \left (h_{1}, \, \Bigl(r_{21}^{\prime}, \, 0, \, r_{23}^{\prime}\Bigr)^{p^{n - k_{11}}} \right )$. Hence,
\begin{equation} \label{e:MuFirst}
\mu(2) = \min \left\{n + k_{13} + k^{(2)}_{1} - k_{11}, \, n + k^{(2)}_3 \right\} - k_{11} \leq n + k^{(2)}_3 - k_{11}\leq n - 1 .
\end{equation}
\item Assume next that $r_{21}^{\prime} = 0$ but $r_{23}^{\prime} \neq 0$. Write $r_{23}^{\prime} = r_{23}^{\prime \prime}p^{k^{(2)}_{3}} \ne 0$ where $k^{(2)}_{3} \geq 0$ and $p \nmid r_{23}^{\prime \prime}.$ Using the standard form for writing elements in $\h(p^n)$, together with the statement before Lemma~\ref{l:EquivCentPseudo}, we obtain
\[
\ \ \ \ \ \ \mu(2) = \mu \left(h_{1}, \, \Bigl(0, \, 0, \, r_{23}^{\prime \prime}p^{k^{(2)}_{3}}\Bigr)^{p^{n - k_{11}}} \right) = \min\left\{n + k_{13}, \, n + k^{(2)}_{3}\right\} - k_{11}.
\]
Since $r^{\prime}_{23}\ne 0$, $n + k^{(2)}_3 - k_{11} < n$. Furthermore, $\nu(h_1, \, h_2)=\nu(h_1)=k_{11}$ implies that $k^{(2)}_{3}<k_{11} \leq k_{13}$.
Therefore,
\begin{equation} \label{e:MuSecond}
\mu(2) = n + k^{(2)}_3 - k_{11} \leq n - 1.
\end{equation}
\item Finally, consider the case when  $r_{23}^{\prime}=0$, but $r_{21}^{\prime} = r_{21}^{\prime \prime} p^{k^{(2)}_1} \neq 0$ where $k^{(2)}_{1} \geq 0$ and $p \nmid r_{21}^{\prime \prime}.$ We compute as before:
\begin{eqnarray*}
\mu(2) & = & \mu \left(h_{1}, \, \Bigl(r_{21}^{\prime \prime}p^{k^{(2)}_{1}}, \, 0, \, 0\Bigr)^{p^{n - k_{11}}} \right)\\
       & = & \min\left\{k_{13} + k^{(2)}_{1} + n - k_{11}, \, k_{11} + n\right\} - k_{11}.
\end{eqnarray*}
Since $h_{2}$ commutes properly with $h_{1}$,
\begin{eqnarray*}
\varphi \left(\left[\left(r_{21}^{\prime \prime}p^{k^{(2)}_1}, \, 0 , \, 0\right), \, {}_{\nu}h_{1} \right] \Big |_2 \right) & = & \varphi\left(r_{21}^{\prime \prime}r_{13} p^{k^{(2)}_{1} + k_{13} - k_{11}}\right)\\
 & = & k^{(2)}_{1} + k_{13} - k_{11} < k_{11}
\end{eqnarray*}
(see Notation \ref{n:Phi}). We conclude that
\begin{equation} \label{e:MuThird}
\mu(2) = n + k_{13} + k^{(2)}_{1} - 2k_{11} < k_{11} + n - k_{11} = n,
\end{equation}
proving our claim that $\mu(2) \leq n-1$ in all situations.
\end{itemize}
Using the expression for $\mu(2)$ from (\ref{e:MuFirst}),
we define:
\[
\widehat{r}_{2}=
\begin{cases}
r_{13}r_{21}^{\prime\prime}p^{n - 2k_{11} + k_{13} + k^{(2)}_{1} - \mu(2)} - r_{11}r_{23}^{\prime\prime}p^{n - k_{11} + k^{(2)}_{3} - \mu(2)} & \text{ if } \, r_{21}^{\prime} \ne 0, \, r_{23}^{\prime} \ne 0;\\
-r_{11}r_{23}^{\prime\prime} & \text{ if } \,  r_{21}^{\prime} = 0, \, r_{23}^{\prime} \ne 0;\\
r_{13}r_{21}^{\prime\prime} & \text{ if } \,  r_{23}^{\prime}= 0, \, r_{21}^{\prime}\ne 0.
\end{cases}
\]
Notice that $\widehat{r}_{2}$ is also a unit in $\mathbb Z_{p^n}$ by definition of $\mu(2)$. Once again, we will use this unit to construct $z_{1}$ and $z_{2}$.

We define $z_{i}$ for $i = 1, \, 2$ as in \textbf{\underline{CASE I}} (see (\ref{zi})). But this time,
\begin{equation} \label{w02}
w_{i} =\widehat{r}_{2}^{\, -1} \Bigl[(i - 1) - (2 - i)r_{21}^{\prime}r_{11}^{-1}p^{n - 2k_{11}}\Bigr]p^{n - 1-\mu(2)}.
\end{equation}

Once again, Lemma \ref{representation1} and Remark \ref{r:Centralizer} imply that $z_{1} \in P(h_1) - C(h_1)$ and $z_2 \in C(h_1)$. The remainder of the proof consists of showing that $z_{2} \in P(h_2) - C(h_2)$ and $z_1 \in C(h_2)$.

In the following calculation we invoke Lemmas~\ref{producttosum} and \ref{l:EquivCentPseudo}. We also use the explicit descriptions of $\mu(2)$ given by (\ref{e:MuFirst}), (\ref{e:MuSecond}), and (\ref{e:MuThird}), along with the equation:
\begin{equation} \label{e:HatMu}
\widehat{r}_{2}p^{\mu(2)} = p^{n - k_{11}}r_{21}^{\prime}r_{13}p^{k_{13} - k_{11}} - p^{n - k_{11}}r_{23}^{\prime}r_{11}.
\end{equation}
This equation is easily obtained from said descriptions of $\mu(2)$. For $i = 1, \, 2,$
we have:   \\
\noindent $\bigl[h_2, \, z_i\bigr] \bigl |_2 = \Bigl[\Bigl(r_{21}^{\prime}, \, 0, \, r_{23}^{\prime}\Bigr)^{p^{n - k_{11}}}, \, {}_{\nu}h_1^{w_i}\bigl(0, \, 0, \, (2 - i)r_{11}^{-1}p^{n - k_{11} - 1}\bigr)\Bigr] \Bigl |_2$

\vspace{.075in}

\noindent $= \Bigl[\Bigl(r_{21}^{\prime}, \, 0, \, r_{23}^{\prime}\Bigr)^{p^{n - k_{11}}}, \, {}_{\nu}h_1^{w_i}\Bigr] \Bigl |_2 + \Bigl[\Bigl(r_{21}^{\prime}, \, 0, \, r_{23}^{\prime}\Bigr)^{p^{n - k_{11}}}, \, \bigl(0, \, 0, \, (2 - i)r_{11}^{-1}p^{n - k_{11} - 1}\bigr)\Bigr] \Bigl |_2$

\vspace{.075in}

\noindent $= w_i\Bigl[\Bigl(r_{21}^{\prime}, \, 0, \, r_{23}^{\prime}\Bigr)^{p^{n - k_{11}}}, \, {}_{\nu}h_1\Bigr] \Bigl |_2 + (2 - i)p^{n-k_{11}}\Bigl[\Bigl(r_{21}^{\prime}, \, 0, \, r_{23}^{\prime}\Bigr), \, \bigl(0, \, 0, \, r_{11}^{-1}p^{n - k_{11} - 1}\bigr)\Bigr] \Bigl |_2$

\vspace{.075in}

\noindent $= w_i\widehat{r}_{2}p^{\mu(2)}+ (2 - i)r_{21}^{\prime}r_{11}^{-1}p^{2(n - k_{11}) - 1}$ \, \,    (by (\ref{e:HatMu}))

\vspace{.075in}

\noindent $= \widehat{r}_{2}^{\, -1} \Bigl[(i - 1) - (2 - i)r_{21}^{\prime}r_{11}^{-1}p^{n - 2k_{11}}\Bigr]p^{n - 1-\mu(2)}\widehat{r}_{2}p^{\mu(2)}+ (2 - i)r_{21}^{\prime}r_{11}^{-1}p^{2(n - k_{11}) - 1}$

\vspace{.075in}

\noindent $=\Bigl[(i - 1) - (2 - i)r_{21}^{\prime}r_{11}^{-1}p^{n - 2k_{11}}\Bigr]p^{n - 1}+ (2 - i)r_{21}^{\prime}r_{11}^{-1}p^{2(n - k_{11}) - 1}$

\vspace{.075in}

\noindent $= (i - 1)p^{n - 1} - (2 - i)r_{21}^{\prime}r_{11}^{-1}p^{2(n - k_{11})-1}+ (2 - i)r_{21}^{\prime}r_{11}^{-1}p^{2(n - k_{11}) - 1}$

\vspace{.075in}

\noindent $= (i - 1)p^{n - 1}.$

\vspace{.075in}

\noindent This immediately gives that $z_{2} \in P(h_2) - C(h_2)$ and $z_1 \in C(h_2)$. The proof of the theorem is now complete.
\end{proof}

\begin{Definition}\label{d:InjectiveSet1}
Let $S$ be a special generating set for a subgroup $H$ of $\hb{n}.$ A set $I$ is termed an \emph{injective set} for $S$ if it is a subset of $S^{\dagger}$ of maximal cardinality satisfying the following conditions:
\begin{enumerate}
\item $|I| \le 2$;
\item If $\left|S^{\dagger}\right| > 0,$ then $\nu(I) = \nu\left(S^{\dagger}\right) = \nu(h_1)$ for some $h_{1} \in S^{\dagger}$;
\item If $I = \{h_1, \, h_2\},$ then $\mu(h_1, \, h_2) \le \mu(h_1, \, h_j)$ for all $h_j \in S^{\dagger} - \{h_1\}$.
\end{enumerate}
\end{Definition}

\begin{Lemma}\label{l:ContainsInejectiveSet}
Every special generating set for a subgroup of $\hb{n}$ contains an injective set.
\end{Lemma}

\begin{proof}
Let $S$ be a special generating set for $H \leq \hb{n}.$ If $H$ is central, then $I = \emptyset$ is an injective set for $S.$ If $\left|S^{\dagger}\right| = 1,$ then $I = S^{\dagger}$ is an injective set for $S.$ Finally, suppose $\left|S^{\dagger}\right| > 1$ and let $h_{1} \in S^{\dagger}$ be such that $\nu\left(S^{\dagger}\right) = \nu(h_{1}).$ Choose $h_2 \in S^{\dagger} - \{h_1\}$ such that $\mu(h_1, \, h_2) \leq \mu(h_1, \, h_j)$ for all $h_j \in S^{\dagger} - \{h_1\}$. Then $I = \{h_1, \, h_2\}$ is an injective set for $S$.
\end{proof}

\begin{Lemma}\label{emptyI}
Suppose that $H \le \hb{n}$ and $I$ is an injective set for any special generating set for $H$. Then $H$ is central if and only if $I = \emptyset.$
\end{Lemma}

\begin{proof}
Let $S$ be any special generating set for $H$. If $H$ is central, then $I = \emptyset$ by the same argument given in the proof of Lemma \ref{l:ContainsInejectiveSet}.

Suppose now that $I = \emptyset$ and $H$ is not central. Then $S^{\dagger}$ contains at least one element, say $h_{1}$. Without loss of generality, we can assume $\nu\left(S^{\dagger}\right) =\nu(h_{1})$. Thus, $h_{1} \in I$ which contradicts the fact that $I$ is empty. Therefore, $H$ is central.
\end{proof}

\begin{Lemma}\label{Isize}
Let $H \le \h(p^{n})$ and $I$ an injective set for a special generating set $S$ for $H.$
\begin{enumerate}
\item If $|S^{\dagger}| = 0,$ then $|I| = 0.$
\item If $|S^{\dagger}| = 1,$ then $|I| = 1.$
\item If $|S^{\dagger}| \geq 2,$ then $|I| = 2.$
\end{enumerate}
\end{Lemma}

\begin{proof}
Let $H \le \h(p^{n})$. Then:
\begin{enumerate}
\item If $|S^{\dagger}| = 0,$ then $H$ is central. The result follows from Lemma \ref{emptyI}.
\item If $|S^{\dagger}| > 0,$ then $H$ is non-central. Since $I \subseteq S^{\dagger},$ we must have $|I| \ge 1$ by Lemma \ref{emptyI}. Thus, if $|S^{\dagger}| = 1,$ then $|I| = 1.$

\item Suppose that $|S^{\dagger}| \ge 2.$ Let $h_{1} \in S^{\dagger}$ such that $\nu(h_{1}) = \nu(S^{\dagger}).$ Then $h_{1}\in I$. Since $S^{\dagger} - \{h_{1}\}\neq \emptyset,$ there exists $h_{2}\in S^{\dagger} - \{h_{1}\}$ such that $\mu(h_{1}, \, h_{2}) \le \mu(h_{1}, \, h^{\prime}_{2})$ for all $h_{2}^{\prime}\in S^{\dagger}-\{h_{1}\}.$ By maximality, we have $I = \{h_{1}, \, h_{2}\}$.
\end{enumerate}
\end{proof}

\begin{Lemma}\label{l:Special}
Let $S$ be a special generating set for $H \leq \h(p^n)$ and $I$ an injective set for $S$ of cardinality 2. Put $I = \{h_1, \, h_2\},$ where $\nu\left(S^{\dagger}\right) = \nu(I) = \nu(h_1)$. If $h_j \in S^{\dagger} - \{h_1\},$ then $h_{j}$ either does not commute with $h_{1}$ or commutes properly with $h_{1}.$
\end{Lemma}

\begin{proof}
This follows directly from Lemma~\ref{l:Properly}.
\end{proof}

The next theorem illustrates the importance for injective sets.

\begin{Theorem}\label{main}
Let $S$ be a special generating set for $H \le \hb{n}$ and $I$ an injective set for $S.$
\begin{enumerate}
\item If $|I| = 0,$ then $|P(H)| = |C(H)|.$
\item If $|I| = 1,$ then $|P(H)| = p|C(H)|.$
\item If $|I| = 2,$ then $|P(H)| = p^{2}|C(H)|.$
\end{enumerate}
\end{Theorem}

Compare (3) with Theorem~\ref{main1}.

\begin{proof}
We invoke Lemma~\ref{P-properties}. If $|I| = 0,$ then $H$ is central and $S$ consists of central elements. Hence, $P(H) = P(S) = \hb{n} = C(H).$

If $|I| = 1,$ then $S^{\dagger} = \{h_{1}\}$ for some $h_{1} \in S$. By Lemma~\ref{l:P(H)=P(SN)} and Theorem~\ref{lo2}, we have $|P(H)| = \left|P\left(S^{\dagger}\right)\right| = |P(h_{1})| = p|C(h_{1})| = p\left|C\left(S^{\dagger}\right)\right| = p|C(H)|.$

Suppose now that $I = \{h_{1}, \, h_{2}\}.$ Throughout the rest of the proof we again adopt the notation from (\ref{Nh}).
By Lemma~\ref{l:Special} we may assume the following:
\begin{itemize}
\item $\nu\left(S^{\dagger}\right) = \nu(I) = \nu(h_1)$;
\item $h_2$ does not commute with $h_1$ or $h_{2}$ commutes properly with $h_{1}$. In the latter case, since $h_{2}$ belongs to a special generating set $S$, by Lemma~\ref{super} (1) we can write $h_2 \sim \Bigl(r_{21}^{\prime}, \, 0, \, r_{23}^{\prime}\Bigr)^{p^{n - \nu(h_1)}}$ where $0 \le r_{21}^{\prime}, \, r_{23}^{\prime} < p^{\nu(h_1)}$ and $r_{21}^{\prime} + r_{23}^{\prime} > 0.$ If $r_{2i}^{\prime} \neq 0$, then we write $r_{2i}^{\prime} = r_{2i}^{\prime \prime} p^{k^{(2)}_i}$ where $k^{(2)}_{i} \geq 0$ and $p \nmid r_{2i}^{\prime \prime}.$
\end{itemize}
Let $\widehat H = \left< I \right> \leq H$. Since $S$ is a special generating set for $H$, $I$ is a special generating set for $\widehat H$ such that $I^{\dagger} = I$ because neither $h_{1}$ nor $h_{2}$ is central. Replacing $S$ with $I$ in Theorem~\ref{main1}, we obtain
\[
P\left(\widehat H\right) / C\left(\widehat H\right)\cong \mathbb Z_p^2.
\]
Let $z_1 C\left(\widehat H\right)$ and $z_2 C\left(\widehat H\right)$ be the generators for $P\left(\widehat H\right) / C\left(\widehat H\right)$ constructed in the proof of Theorem~\ref{main1} (see (\ref{zi}), (\ref{w0}) and (\ref{w02})). In principle, $z_1$ and $z_2$ both belong to $P\left(\widehat H\right)$. The bulk of the proof will consist on showing that $z_1$ and $z_2$ are, in fact, in the potentially smaller group $P(H)$. Assume this has been done and $z_{1}$ and $z_{2}$ belong to $P(H)$. It follows from Lemma~\ref{diagramChasing} that $P(H)/C(H)$ is isomorphic to $\mathbb Z_p$ or $\mathbb Z_p^2.$ By Lemma~\ref{P-properties}, the map
\[
\psi: P(H)/C(H) \rightarrow P\left(\widehat H\right) / C\left(\widehat H\right)
 \]
given by $xC(H) \mapsto xC\left(\widehat H\right)$ is a well defined homomorphism. Notice that $z_1 C(H)$ and $z_2 C(H)$ are distinct and non-trivial in $P(H)/C(H)$. If $P(H)/C(H)$ were isomorphic to $\mathbb Z_p$, one of these elements would be a power of the other. But this would also hold for $z_1 C\left(\widehat H\right)$ and $z_2 C\left(\widehat H\right)$, which it does not. We conclude that $\psi$ is an isomorphism and $P(H)/C(H)$ is isomorphic to $\mathbb Z_p^2$.

The possibly different values for $w_i$ (see (\ref{w0}) and (\ref{w02})) will not play a role at first. We show then that $z_{1}, \, z_{2} \in P(h_{j})$ for all $h_{j} \in S^{\dagger}.$ Once this is established, Lemmas~\ref{P-properties} (5, 9) and \ref{l:P(H)=P(SN)} will give $z_{1}, \, z_{2} \in P\left(S^{\dagger}\right) = P(H),$ completing the proof of the theorem.

Let $\mu(j) = \mu(h_{1}, \, h_{j})$. For simplicity, we assume that $\nu(h_{1}) = k_{11}.$  We already know that $z_{1}, \, z_{2} \in P(h_{j})$ for $j = 1, \, 2$. In our generalized argument we will omit the case $j = 1$ but still include the case $j = 2$.
We define the following two sets:
\begin{eqnarray*}
S_1 & = & \left\{h \in S^{\dagger}  \, \middle| \,  h \text{ does not commute with $h_1$}\right\} \hbox{ \ \ \ and}\\
S_2 & = & \left\{h \in S^{\dagger} \, | \, h \neq h_{1} \text{ and } h \text{ commutes properly with $h_1$}\right\}.
\end{eqnarray*}
Suppose first that $h_j \in S_{1}$ ($j >2$). Using Lemma~\ref{producttosum}, we have (for $i = 1, \, 2$):

$[h_{j}, \, z_{i}]|_{2} = \Bigl[h_{j}, \, {}_{\nu}h_{1}^{w_{i}}\Bigl(0, \, 0, \, (2 - i)r_{11}^{-1}p^{n - k_{11} - 1}\Bigr)\Bigr] \Bigl |_{2}$

\vspace{.075in}

$\ \ \ = \Bigl[h_{j}, \, {}_{\nu}h_{1}^{w_{i}}\Bigr] \Bigl |_{2} + \Bigl[h_{j}, \, \Bigl(0, \, 0, \, (2 - i)r_{11}^{-1}p^{n - k_{11} - 1}\Bigr)\Bigr] \Bigl |_{2}$

\vspace{.075in}

$\ \ \ = \bigl(w_i \bigl(r_{j1}r_{13}p^{k_{13} + k_{j1} - k_{11}} - r_{11}r_{j3}p^{k_{j3}} \bigr) + (2 - i)r_{11}^{-1}r_{j1}p^{n + k_{j1} - k_{11} - 1}  \bigr)$

\vspace{.075in}

$\ \ \ = \Bigl(w_{i}r_{j}p^{\mu(j)} + (2 - i)r_{11}^{-1}r_{j1}p^{n + k_{j1} - k_{11} - 1}\Bigr) \hbox{ mod } p^{n},$

\vspace{.075in}

\noindent where, by definition of $\mu(j)$,
\[
r_j = r_{13}r_{j1}p^{k_{13} + k_{j1} - k_{11} - \mu(j)} - r_{11}r_{j3}p^{k_{j3} - \mu(j)}
\]
is a unit in $\Z_{p^{n}}$. Since $\nu\left(S^{\dagger}\right) = k_{11} \leq k_{j1}$, we observe immediately that
\begin{equation} \label{e:Pseudo1}
n - 1 \leq n + k_{j1} - k_{11} - 1.
\end{equation}
We argue next that $\varphi \left(w_{i} r_{j} p^{\mu(j)} \right) \geq n - 1$ for $i = 1, \, 2$ (see Notation \ref{n:Phi}). Indeed, using the explicit descriptions for $w_{1}$ and $w_{2}$ given by (\ref{w0}) and (\ref{w02}), a direct computation gives:
\small
\begin{align*}
&w_{i} r_{j} p^{\mu(j)} =\\
&
\begin{cases}
r_{j}r_{2}^{-1}\Bigl[(i - 1) - (2 - i)r_{11}^{-1}r_{21}p^{k_{21} - k_{11}}\Bigr]p^{n + \mu(j) - \mu(2) - 1} &\hbox{if \ \ } h_2\in S_1;\\
r_{j}\widehat{r}_2^{\, -1}\Bigl[(i - 1) - (2 - i)r_{11}^{-1}r_{21}^{\prime\prime}p^{n - 2k_{11} + k^{(2)}_{1}}\Bigr]p^{n + \mu(j) - \mu(2) - 1} &\hbox{if \ \ } h_2\in S_2, \, r_{21}^\prime \ne 0;\\
r_{j} \widehat{r}_2^{\, -1}\left(i - 1\right)p^{n + \mu(j) - \mu(2) - 1} &\hbox{if \ \ } h_2\in S_2, \, r_{21}^\prime = 0.\\
\end{cases}
\end{align*}
\normalsize
Thus we have $\varphi(w_{i} r_{j} p^{\mu(j)})\ge n + \mu(j) - \mu(2) - 1$ (we interpret $0$ as $p^{n}$). Since $\nu(S^{\dagger}) = k_{11}$ and $I$ is an injective set, we have $\mu(j) \geq \mu(2)$. From this it follows that
\[
n - 1 \le n + \mu(j) - \mu(2) - 1 \le \varphi\left(w_{i} r_{j} p^{\mu(j)}\right).
\]
Using (\ref{e:Pseudo1}) we deduce that $z_1, \, z_2 \in P(h_{j})$, as promised.

Suppose next that $h_j \in S_{2}$ ($j> 2$). Then $h_j \sim \Bigl(r_{j1}^{\prime}, \, 0, \, r_{j3}^{\prime}\Bigr)^{p^{n - k_{11}}},$ where $0 \le r_{j1}^{\prime}, \, r_{j3}^{\prime} <p^{k_{11}}$ and $r_{j1}^{\prime} + r_{j3}^{\prime} > 0$. We have three subcases to consider:

\vspace{.1in}

\noindent\underline{\textbf{Subcase I:}} Suppose that $r_{j1}^{\prime}, \, r_{j3}^{\prime} > 0.$ Put $r^{\prime}_{j1} = r^{\prime\prime}_{j1}p^{k_1^{(j)}}$ where $k_1^{(j)} \geq 0$ and $p\nmid r^{\prime\prime}_{j1},$ and $r^{\prime}_{j3} = r^{\prime \prime}_{j3} p^{k_3^{(j)}}$ where $k_3^{(j)} \geq 0$ and $p\nmid r^{\prime\prime}_{j3}$. By Lemmas~\ref{producttosum} and \ref{l:EquivCentPseudo}, we have (for $i = 1, \, 2$):
\begin{DispWithArrows}[tagged-lines = last]
[h_j, \, z_i]|_2 & =  \Bigl[\Bigl(r^{\prime \prime}_{j1} p^{k_1^{(j)}}, \, 0, \, r_{j3}^{\prime \prime} p^{k_3^{(j)}}\Bigr)^{p^{n - k_{11}}}, \, {}_{\nu}h_{1}^{w_i}\Bigl(0, \, 0, \, (2 - i)r_{11}^{-1}p^{n - k_{11} - 1}\Bigr)\Bigr] \Bigl|_2\\
              & =  \Bigl[\Bigl(r^{\prime \prime}_{j1} p^{k_1^{(j)}}, \, 0, \, r_{j3}^{\prime \prime} p^{k_3^{(j)}}\Bigr)^{p^{n - k_{11}}}, \, {}_{\nu}h_{1}^{w_i}\Bigr]\Bigl |_2\\
              & \qquad +  \Bigl[\Bigl(r^{\prime \prime}_{j1} p^{k_1^{(j)}}, \, 0, \, r_{j3}^{\prime \prime} p^{k_3^{(j)}}\Bigr)^{p^{n - k_{11}}}, \, \Bigl(0, \, 0, \, (2 - i)r_{11}^{-1}p^{n - k_{11} - 1}\Bigr)\Bigr] \Bigl |_2\\
              & =\Bigl(w_i\Bigl[\Bigl(r^{'}_{j1}, \, 0, \, r_{j3}^{'}\Bigr)^{p^{n - k_{11}}}, \, {}_{\nu}h_1\Bigr] \Bigl |_2+ (2 - i)r_{j1}^{''}r_{11}^{-1}p^{2(n - k_{11}) - 1 + k_{1}}\Bigr) \\
              & =  \Bigl( w_i\widehat{r}_{j}p^{\mu(j)} + (2 - i)r_{j1}^{\prime \prime}r_{11}^{-1}p^{2(n - k_{11}) - 1 + k_{1}^{(j)}}\Bigr) \hbox{ mod } p^{n},\label{cso1}
\end{DispWithArrows}
\noindent where
\[
\widehat{r}_{j}=r_{j1}^{\prime \prime}r_{13}p^{n+k_{1}^{(j)}+k_{13}-2k_{11}-\mu(j)}
-r_{j3}^{\prime \prime}r_{11}p^{n+k_{3}^{(j)}-k_{11}-\mu(j)}.
\]
By the statement before Lemma~\ref{l:EquivCentPseudo},
\begin{equation}
\begin{split}
\mu(j) & = \mu\left(h_{1}, \, \left(r^{\prime \prime}_{j1}p^{k_{1}^{(j)}}, \, 0, \, r^{\prime \prime}_{j3}p^{k_{3}^{(j)}} \right)^{p^{n - k_{11}}}\right)\\
       & = \min\left\{n + k_{1}^{(j)} + k_{13} - k_{11}, \, n + k_{3}^{(j)}\right\} - k_{11}.
\end{split}
\end{equation}
We conclude that $\widehat{r}_{j}$ is a well-defined unit in $\mathbb Z_{p^n}$. We must show that the powers of $p$ obtained in the above expression for $[h_{j}, \, z_{i}]|_{2}$ are at least $n - 1$. By Lemma~\ref{super} (1), $2k_{11} \leq n + k_{1}^{(j)}$. Thus,
\begin{equation} \label{e:Ineq}
n - 1 \leq 2(n - k_{11}) + k_1^{(j)} - 1.
\end{equation}
\noindent From this we can deduce that the power of $p$ in the second term of (\ref{cso1}) is at least $n - 1$. We will prove that the same holds for the first term of (\ref{cso1}) at the end of the proof.

\vspace{.1in}

\noindent\underline{\textbf{Subcase II:}} Assume next that $r^{\prime}_{j1} > 0$ and $r^{\prime}_{j3} = 0.$ As before, we write $r^{\prime}_{j1} = r^{\prime \prime}_{j1}p^{k_1^{(j)}}$ with $k_1^{(j)} \geq 0$ and $p \nmid r^{\prime \prime}_{j1}$. By Lemmas~\ref{producttosum} and \ref{l:EquivCentPseudo} we have (for $i = 1, \, 2$):
\begin{DispWithArrows}[tagged-lines = last]
[h_j, \, z_i]|_2 & = \Bigl[\Bigl(r^{\prime \prime}_{j1} p^{k_1^{(j)}}, \, 0, \, 0 \Bigr)^{p^{n - k_{11}}}, \, {}_{\nu}h_{1}^{w_i}\Bigl(0, \, 0, \, (2 - i)r_{11}^{-1}p^{n - k_{11} - 1}\Bigr)\Bigr] \Bigl|_2\\
                 & = \Bigl[\Bigl(r^{\prime \prime}_{j1} p^{k_1^{(j)}}, \, 0, \,  0\Bigr)^{p^{n - k_{11}}}, \, {}_{\nu}h_{1}^{w_i}\Bigr]\Bigl |_2\\
                 & \qquad + \Bigl[\Bigl(r^{\prime \prime}_{j1} p^{k_1^{(j)}}, \, 0, \, 0 \Bigr)^{p^{n - k_{11}}}, \, \Bigl(0, \, 0, \, (2 - i)r_{11}^{-1}p^{n - k_{11} - 1}\Bigr)\Bigr] \Bigl |_2\\
                 & =\Bigl(w_i\Bigl[\Bigl(r^{'}_{j1}, \, 0, \, r_{j3}^{'}\Bigr)^{p^{n - k_{11}}}, \, {}_{\nu}h_1\Bigr] \Bigl |_2 + (2 - i)r_{j1}^{''}r_{11}^{-1}p^{2(n - k_{11}) - 1 + k_{1}}\Bigr) \\
                 & =\Bigl( w_i\widehat{r}_{j}p^{\mu(j)} + (2 - i)r_{j1}^{\prime \prime}r_{11}^{-1}p^{2(n - k_{11}) - 1 + k_{1}^{(j)}}\Bigr) \hbox{ mod } p^{n},\label{cso2}
\end{DispWithArrows}
where
\[
\widehat{r}_j = r_{j1}^{\prime \prime}r_{13}p^{n + k_{1}^{(j)} + k_{13} - 2k_{11} - \mu(j)}.
\]
Since $h_{j}$ commutes properly with $h_{1}$, the same calculation that led to (\ref{e:MuThird}) gives that $\mu(j) = n + k_{13} + k^{(j)}_{1} - 2k_{11}.$ We conclude that $\widehat{r}_j$ is, once again, a well-defined unit. Again, we can deduce that the exponent of $p$ in the second term of (\ref{cso2}) is at least $n-1$. We will prove that the same is true for the first term of (\ref{cso2}) at the end of the proof.

\vspace{.1in}

\noindent\underline{\textbf{Subcase III:}} Finally, assume that $r^{\prime}_{j1} = 0$ and $r^{\prime}_{j3} > 0$. As usual, write $r^{\prime}_{j3} = r^{\prime \prime}_{j3} p^{k^{(j)}_3}$ with $k^{(j)}_3 \geq 0$ and $p \nmid r^{\prime \prime}_{j3}$. Using Lemmas~\ref{producttosum} and \ref{l:EquivCentPseudo} and computing in the standard way, we have (for $i = 1, \, 2$):
\begin{DispWithArrows}[tagged-lines = last]
[h_j, \, z_i]|_2 & = \Bigl[\Bigl(0, \, 0, \, r_{j3}^{\prime \prime} p^{k^{(j)}_3}\Bigr)^{p^{n - k_{11}}}, \, {}_{\nu}h_{1}^{w_i}\Bigl(0, \, 0, \, (2 - i)r_{11}^{-1}p^{n - k_{11} - 1}\Bigr)\Bigr] \Bigl|_2\\
                 & = \Bigl[\Bigl(0, \, 0, \, r_{j3}^{\prime \prime} p^{k^{(j)}_3}\Bigr)^{p^{n - k_{11}}}, \, {}_{\nu}h_{1}^{w_i}\Bigr]\Bigl |_2\\
                 & \qquad + \Bigl[\Bigl(0, \, 0, \, r_{j3}^{\prime \prime} p^{k^{(j)}_3}\Bigr)^{p^{n - k_{11}}}, \, \Bigl(0, \, 0, \, (2 - i)r_{11}^{-1}p^{n - k_{11} - 1}\Bigr)\Bigr] \Bigl |_2\\
                 & = \Bigl(w_i\Bigl[\Bigl(r^{'}_{j1}, \, 0, \, r_{j3}^{'}\Bigr)^{p^{n - k_{11}}}, \, {}_{\nu}h_1\Bigr] \Bigl |_2 + (2 - i)r_{j1}^{''}r_{11}^{-1}p^{2(n - k_{11}) - 1 + k_{1}}\Bigr) \\
                 & = w_i\widehat{r}_{j}p^{\mu(j)} \hbox{ mod } p^{n}, \label{cso3}
\end{DispWithArrows}
where $\widehat{r}_{j} = -r^{\prime \prime}_{j3}r_{11}p^{n + k^{(j)}_3 - k_{11} - \mu(j)}$. By the statement before Lemma~\ref{l:EquivCentPseudo} we have:
\[
\mu(j) = \mu\left(h_{1}, \, \Bigl(0, \, 0, \, r_{j3}^{\prime \prime}p^{k^{(j)}_{3}}\Bigr)^{p^{n - k_{11}}} \right) = \min\left\{n + k_{13}, \, n + k^{(j)}_{3}\right\} - k_{11}.
\]
Since $r^{\prime}_{j3} \neq 0$, we have that $n + k^{(j)}_3 - k_{11} < n$, and hence, $k^{(j)}_{3} < k_{11}$. Moreover, $\nu(h_{1}) = k_{11} \leq k_{13}$ readily implies that $\mu(j) = n + k^{(j)}_{3} - k_{11}$ in this case. We again conclude that $\widehat{r}_{j}$ is a well-defined unit.

To finish the proof, we calculate the power of $p$ in the first term of (\ref{cso1}), (\ref{cso2}), and in the only term of (\ref{cso3}). We have:
\small
\begin{align*}
&w_{i} \widehat{r}_{j} p^{\mu(j)} =\\
&
\begin{cases}
\widehat{r}_{j}r_{2}^{-1}\Bigl[(i - 1) - (2 - i)r_{11}^{-1}r_{21}p^{k_{21} - k_{11}}\Bigr]p^{n + \mu(j) - \mu(2) - 1} &\hbox{if \ \ \ } h_2 \in S_1;\\
\widehat{r}_{j}\widehat{r}_2^{\, -1}\Bigl[(i - 1) - (2 - i)r_{11}^{-1}r_{21}^{\prime\prime}p^{n - 2k_{11} + k^{(2)}_{1}}\Bigr]p^{n + \mu(j) - \mu(2) - 1} &\hbox{if \ \ \ }h_2\in S_2, \, r_{21}^{\prime} \ne 0;\\
\widehat{r}_{j} \widehat{r}_2^{\, -1}\left(i - 1\right)p^{n + \mu(j) - \mu(2) - 1} &\hbox{if \ \ \ } h_2\in S_2,\, r_{21}^{\prime} = 0.\\
\end{cases}
\end{align*}
\normalsize
Thus, we have $\varphi\left(w_{i} \widehat{r}_{j}p^{\mu(j)}\right) \ge n + \mu(j) - \mu(2) - 1$ (we interpret $0$ as $p^{n}$). Since $\nu\left(S^{\dagger}\right) = k_{11}$ and $I$ is an injective set, we have $\mu(j) \geq \mu(2)$. From this it follows that
\[
n - 1 \le n + \mu(j) - \mu(2) - 1 \le \varphi\left(w_{i} \widehat{r}_{j} p^{\mu(j)}\right).
\]
By (2) of Corollary~\ref{P}, $w_i\widehat{r}_jp^{\mu(j)}$ belongs to $P(h_{j})$ in all situations. From this we deduce that $z_1, \, z_2 \in P(h_{j})$, as promised.
\end{proof}

\begin{Corollary}\label{c:InjectiveForSubgroup}
Let $H \le \hb{n}$ with special generating sets $S_1$ and $S_2.$ If $I_1$ and $I_2$ are injective sets for $S_{1}$ and $S_{2}$ respectively, then $|I_1| = |I_2|$.
\end{Corollary}

In light of Corollary~\ref{c:InjectiveForSubgroup}, we can now generalize Definition~\ref{d:InjectiveSet1}.

\begin{Definition}\label{d:InjectiveSetForSubgroup}
An \emph{injective set} for any subgroup $H$ of $\hb{n}$ is an injective set for some special generating set for $H$.
\end{Definition}

\begin{Remark}\label{re-in}
By Corollary~\ref{c:InjectiveForSubgroup}, every injective set of a subgroup of $\hb{n}$ has the same cardinality.
\end{Remark}

We now digress and complete the proof of Lemma~\ref{PnoC2} using Theorem~\ref{main}. Suppose that $H \leq \h(p^n)$ and $P(H) = C(H).$ We assert that $H \leq Z(\h(p^n)).$ Let $I$ be any injective set for $H.$ We claim that $I$ must be empty. If $I \neq \emptyset,$ then $|P(H)| = p^{i}|C(H)|$ for $i = 1$ or $i = 2$ by Theorem~\ref{main}. This means that $|P(H)| > |C(H)|$ and, consequently, $P(H) \neq C(H).$ This contradicts the hypothesis that $P(H) = C(H)$ and proves the claim. Hence, by Lemma~\ref{emptyI}, $H \leq Z(\h(p^n))$ as asserted.

The next lemma will be needed later.

\begin{Lemma}\label{inj-size}
Let $q : \hb{n} \rightarrow \hb{n - 1}$ be the group homomorphism from Lemma~\ref{short}. If $H \leq \hb{n - 1}$ and $I$ is an injective set for $q^{-1}(H),$ then $|I| = 2.$
\end{Lemma}

\begin{proof}
Let $S$ be a special generating set for $H$.  We fix $S$, once and for all, for the rest of the proof. For $h \in S$ we let $\widehat{h}$ be the same element as $h$, except viewed as lying in $\h(p^n)$.
We claim that
\[
\widehat{S} = \left\{\widehat{h} \, \big| \, h \in S\right\} \bigcup \left\{\left(p^{n - 1}, \, 0, \, 0\right), \, \left(0, \, 0, \, p^{n - 1}\right), \ \left(0, p^{n - 1}, \, 0\right)\right\}
\]
is a generating set for $q^{-1}(H)$. To see this, let $x \in q^{-1}(H)$. Then $q(x) \in H = \left< S \right >$, so that
 $q(x) = s_{1}^{m_{1}} \cdots s_{l}^{m_{l}}$ for some $s_{1}, \, \ldots, \, s_{l} \in S$ and integers $m_{1}, \, \ldots, \, m_{l}.$ Set $y = {\left(\widehat s_{1} \right)}^{m_{1}} \cdots {\left( \widehat s_{l} \right)}^{m_{l}}$, and note that $q(y) = q(x)$. This means that
\[
xy^{-1} \in \ker q = \hbox{im }f =
\left< \left(p^{n - 1}, \, 0, \, 0\right), \, \left(0, \, 0, \, p^{n - 1}\right), \ \left(0, p^{n - 1}, \, 0\right)  \right>,
\]
proving our claim.

We have two cases to consider:

\noindent \textbf{\underline{CASE I}: $H$ is central.}

In this case,  $H = \left<s\right>$ for some $s \in Z \left(\h \left(p^{n - 1} \right)\right).$ Hence,
$q^{-1}(H)$ is generated by
\[
\widehat{S} = \left\{ \widehat{s}, \,  \left(p^{n - 1}, \, 0, \, 0\right), \, \left(0, \, 0, \, p^{n - 1}\right), \ \left(0, \, p^{n - 1}, \, 0\right) \right\}.
\]
By definition of $\widehat{s}$, both
$\widehat{s}$ and $\left(0, p^{n - 1}, \, 0\right)$ lie in $Z \left(\h(p^n)\right).$ Hence, there exists $z \in Z(\h(p^n))$ such that
\[
q^{-1}(H) = \left< \left(p^{n - 1}, \, 0, \, 0\right), \, z, \, \left(0, \, 0, \, p^{n - 1}\right) \right>.
\]
Notice that $\left \{\left(p^{n - 1}, \, 0, \, 0\right), \, z, \, \left(0, \, 0, \, p^{n - 1}\right) \right \}$ is a special generating set and \\$\left \{ \left(p^{n - 1}, \, 0, \, 0\right), \, \left(0, \, 0, \, p^{n - 1}\right) \right \}$ an injective set for $q^{-1}(H)$. This completes the first case.\\

\noindent \textbf{\underline{CASE II}: $H$ is non-central.}

We handle first the situation when injective sets have two elements. Accordingly, we let $I = \{h_1, \, h_2\} \subseteq S^{\dagger}$ be an injective set for $H$. Without loss of generality, we set $\nu(I) = \nu(S^{\dagger}) = \nu(h_{1})$.
Recall that 
\[
\widehat S = \left\{\widehat{h} \, \vert \, h \in S\right\}
\bigcup \left\{\left(p^{n - 1} , \, 0, \, 0 \right),  \left(0 , \, 0, \, p^{n - 1} \right), \left(0 , \, p^{n - 1}, \, 0 \right)    \right\}
\]
generates $q^{-1}(H)$. Note also that
\[
\widehat{S}^{\dagger} = \left\{\widehat{h} \, \vert \, h \in S    \right\} \bigcup \left\{\left(p^{n - 1} , \, 0, \, 0 \right), \left(0, \, 0, \, p^{n - 1}\right)\right\}.
\]
In particular, observe that $\widehat{h}_1$ and $\widehat{h}_2$ lie in $\widehat{S}^{\dagger}.$ Our goal is to construct a special generating set as outlined in the proof of Lemma~\ref{l:ContainsInejectiveSet}. By construction, for all $h \in S^{\dagger}$ we have
\[
\nu\left(\widehat h_{1}\right) \leq \nu\left(\widehat{h}\right) \leq \nu\left(p^{n - 1}, \, 0, \, 0\right), \, \nu\left(0, \, 0, \, p^{n - 1}\right).
\]
This ensures that $\nu\left(\widehat{S}^{\dagger}\right) = \nu\left(\widehat{h}_{1}\right).$

Now, $\widehat{h}_2 \nsim \left(\widehat{h}_1\right)^{w}$ for any $w > 0.$ For if $\widehat{h}_{2} \sim \left(\widehat{h}_1\right)^{w}$ for some $w > 0,$ then $h_{2} = q\left(\widehat{h}_{2}\right) \sim q\left(\widehat{h}_{1}\right)^{\omega} = h_{1}^{\omega},$ and thus, $S$ would not be special. We proceed as in the proof of Lemma~\ref{l:ContainsInejectiveSet} and obtain a special generating set ${\widehat S}^{\prime}$ containing $\widehat{h}_1$ and $\widehat{h}_2.$ Thus, $\left|\left(\widehat{S}^{\prime}\right)^{\dagger}\right| > 1$.

To finish the proof we consider the case in which $I = \{ h \} \subseteq S^{\dagger}$ is an injective set for $H$ with $\nu(I) = \nu\left( {S^{\dagger}} \right) = \nu (h)$. Again, we follow the proof of Lemma \ref{l:ContainsInejectiveSet} to construct a special generating set $\widehat{S}^{\prime}$ from $\widehat{S}$ as in the first situation. By the same argument as in the previous paragraphs we must have $\widehat h \in \left(\widehat{S}^{\prime}\right)^{\dagger}$. 

Notice that if $h|_{1} = 0,$ then there is no $w > 0$ such that $\left(\widehat{h}\right)^{w} \sim \left(p^{n - 1}, \, 0, \, 0\right)$. Similarly, if $h|_{3} = 0,$ then there is no $w > 0$ such that $\left(\widehat{h}\right)^{w} \sim \left(0, \, 0, \, p^{n - 1}\right).$ This means that the special generating set $\widehat{S}^{\prime}$ obtained from $\widehat{S}$ must contain $\widehat{h}$ and either $\left(p^{n - 1}, \, 0, \, 0\right)$ or $\left(0, \, 0, \, p^{n - 1}\right)$. Thus, $\left|\left(\widehat{S}^{\prime}\right)^{\dagger}\right| \ge 2.$ The result follows from Lemma~\ref{Isize} (3).

Suppose now that $h|_{1}$ and $h|_{3}$ are non-zero and that there are $w_{1} > 0$ and $w_{2} > 0$ such that $\widehat{h}^{w_{1}} \sim \left(p^{n - 1}, \, 0, \, 0\right)$ and $\widehat{h}^{w_{2}} \sim \left(0, \, 0, \, p^{n - 1}\right)$. Using the notation given in (\ref{n:Factored_h}) we must have: 
\begin{multicols}{2}
\begin{enumerate}
\item $r_{1}w_{1}p^{k_{1}} = p^{n - 1}\hbox{ mod }p^{n}$;
\item $r_{3}w_{1}p^{k_{3}} = p^{n}\hbox{ mod }p^{n}$;
\item $r_{3}w_{2}p^{k_{3}} = p^{n - 1}\hbox{ mod }p^{n}$;
\item $r_{1}w_{2}p^{k_{1}} = p^{n}\hbox{ mod }p^{n}$.
\end{enumerate}
\end{multicols}
\noindent From (1) we can deduce that $w_{1} = r_{1}^{-1}p^{n - 1 - k_{1}}$ and substituting in (2) we obtain $k_{3} \ge k_{1} + 1$. Similarly, from (3) we obtain $w_{2} = r_{3}^{-1}p^{n - 1 - k_{3}}$ and substituting in (4) we obtain $k_{1} \ge k_{3} + 1$. From this we get $k_{1} \ge k_{1} + 2$ which is absurd. Thus, we must have either $\widehat{h}^{w} \nsim \left(p^{n - 1}, \, 0, \, 0\right)$ or $\widehat{h}^{w} \nsim \left(0, \, 0, \, p^{n - 1}\right)$ for all $w>0$. 

We conclude that the special generating set $\widehat{S}^{\prime}$  must contain $\widehat{h}$ and either $\left(p^{n - 1}, \, 0, \, 0\right)$ or $\left(0, \, 0, \, p^{n - 1}\right)$. Thus $\left|\left(\widehat{S}^{\prime}\right)^{\dagger}\right|\ge 2$ and this gives the result by Lemma~\ref{Isize}  (3).
\end{proof}

Next we define a function that will allow us to consolidate some of our earlier results. By Remark~\ref{re-in}, this function is well-defined.

\begin{Definition}\label{d:Delta}
For any $H \le \h(p^{n})$ and any injective set $I$ for $H,$ we define:
\[
\delta(H) = 3 - |I|
\]
\end{Definition}

\begin{Lemma}\label{post-main}
If $H \le \h(p^{n}),$ then $|C(H)| = p^{\delta(H)}|C(q(H))|.$
\end{Lemma}

\begin{proof}
The cases when $\delta(H) = 2$ and $3$ follow from Corollary \ref{premain2} and Lemma~\ref{emptyI}. If $\delta(H) = 1,$ then it follows from (4) of Lemma \ref{P-properties} and Theorem \ref{main} that
\[
p^{2}|C(H)| = |P(H)| = |\ker q||C(q(H))| = \left|\mathbb{Z}_{p}^{3}\right|\left|C(q(H))\right| = p^{3}|C(q(H))|.
\]
\noindent This gives the result.
\end{proof}

\section{The chermak-delgado and pseudo chermak-delgado measure}

\begin{Definition}
The \emph{Chermak-Delgado measure} of a subgroup $H$ of a finite group $G$ is $m(H) = |H||C(H)|.$ The maximum of the set $\{m(H) \, | \, H \le G\}$ is denoted by $m^{\ast}(G)$ and termed the \emph{Chermak-Delgado measure} of $G$ (see \cite{CD} for details).
\end{Definition}

Our goal in this section is to find $m^{\ast}(\h(p^{n})).$ This will be achieved by studying the relationship between $m^{\ast}(\h(p^{n}))$ and the so-called \emph{pseudo Chermak-Delgado measure} of $\h(p^{n}).$ Before we introduce this new concept, we prove a useful lemma which involves a short exact sequence from the earlier part of this paper.

\begin{Lemma}\label{pre-p-n}
For the short exact sequence (\ref{sh-H}) and $H \le \h \left(p^{n - 1}\right),$ we have $$m\left(q^{-1}(H)\right) = p^{4}m(H).$$
\end{Lemma}

\begin{proof}
Let $K = q^{-1}(H)$. By Theorem~\ref{L1}, $|K| = |K \cap \ker q||q(K)|.$ Now, $\ker q = \hbox{im }f,$ $\hbox{ im }f \leq K,$ and $|q(K)| = |H|.$ Thus,
\begin{equation}\label{e:pre-p-n}
|K| = |K \cap \hbox{im }f||H| = |\hbox{im }f||H| = |\ker q||H| = p^{3}|H|.
\end{equation}
Using Lemma \ref{post-main}, we have
\[
m(K) = |K||C(K)| = p^{3}|H||C(K)| = p^{3}|H|p^{\delta(K)}|C(H)| = p^{3 + \delta(K)}m(H).
\]
By Lemma~\ref{inj-size}, any injective set for $q^{-1}(H)$ has cardinality $2.$ Thus, $\delta(K) = 1$ and the result follows.
\end{proof}

\begin{Definition}
For a short exact sequence (\ref{shes}) and $H \le G$, the \emph{pseudo Chermak-Delgado measure} of $H$, with respect to $q$, denoted by $m_{s}(H; \, q),$
is defined as
\[
m_{s}(H; \, q) = |H||P(H; \, q)|.
\]
The maximum of the set $\{m_{s}(H; \, q) \, | \, H \le G\}$ is denoted by $m^{\ast}_{s}(G; \, q).$
\end{Definition}

\begin{Notation}
If the short exact sequence is fixed, then we write $m_{s}(H)$ and $m^{\ast}_{s}(G)$ for $m_{s}(H; \, q)$ and $m^{\ast}_{s}(G; \, q)$ respectively.
\end{Notation}

The next few results relate the two measures in certain instances.

\begin{Lemma}\label{compare}
For any short exact sequence (\ref{shes}) and $H \le G,$ we have
\[
m_{s}(H) \ge m(H) \hbox{ \ and \ } m_{s}^{\ast}(G) \ge m^{\ast}(G).
\]
\end{Lemma}

\begin{proof}
It follows from Lemma~\ref{P-properties} (1).
\end{proof}

\begin{Lemma}\label{relation1}
For any short exact sequence (\ref{shes}) and $H \le G,$ we have
\[
m_{s}(H) = m(q(H))|G_{1}||\hbox{im } f \cap H|.
\]
\end{Lemma}

\begin{proof}
By Theorem~\ref{L1}, we have $|H| = |\ker q \cap H||q(H)| = |\hbox{im }f \cap H||q(H)|.$ Using Lemma~\ref{P-properties} (4), we get $m_{s}(H) = |H||P(H)| = |\hbox{im } f \cap H||q(H)||G_{1}||C(q(H))| = m(q(H))|G_{1}||\hbox{im } f \cap H|.$
\end{proof}

\begin{Corollary}\label{relation2}
For any short exact sequence (\ref{shes}) and $H \le G_{2},$ we have
\[
m_{s}\left(q^{-1}(H)\right) = m(H)|G_{1}|^{2}.
\]
\end{Corollary}

\begin{proof}
It follows from Lemma \ref{relation1} and the fact that $\hbox{im } f \le q^{-1}(H).$
\end{proof}

\begin{Lemma}\label{p-n}
For the short exact sequence (\ref{sh-H}) and any $H \le\h(p^{n}),$ we have
\[
m_{s}(H) = p^{3-\delta(H)}m(H).
\]
\end{Lemma}

\begin{proof}
By Lemmas \ref{P-properties} (4) and \ref{post-main}, we have

$m_{s}(H) = |H||P(H)| = |H|\left|\Z_{p}^{3}\right||C(q(H))| = |H|p^{3}p^{-\delta(H)}|C(H)| = p^{3-\delta(H)}m(H).$
\end{proof}

\begin{Lemma}\label{general-p-n}
For the short exact sequence (\ref{sh-H}) and $H \le \h \left(p^{n - 1}\right),$ we have
\[
m_{s}\left(q^{-1}(H)\right) = p^{6}m(H).
\]
\end{Lemma}

\begin{proof}
By Lemma \ref{P-properties} (4) and Equation (\ref{e:pre-p-n}), we have

$m_{s}\left(q^{-1}(H)\right) = \left|q^{-1}(H)\right|\left|P\left(q^{-1}(H)\right)\right| = p^{3}\left|H\right|\left|\mathbb{Z}_{p}^{3}\right|\left|C(H)\right| = p^{6}m(H).$
\end{proof}

\section{the chermak-delgado and pseudo chermak-delgado lattice}

For any finite group $G$, let $\ch(G) = \{H \le G \, | \, m^{\ast}(G) = m(H) \}$ and, for a fixed short exact sequence (\ref{shes}), $\pc(G) = \{H \le G \, | \, m_{s}^{\ast}(G) = m_{s}(H)\}$ . Suppose that $H, \, K \in \ch(G).$ It is known that $\ch(G)$ is a lattice, called the \emph{Chermak-Delgado Lattice}, with the meet operation defined as $H \wedge K = H \cap K$ and the join operation defined as $H \vee K = HK.$ One can show that $HK = \left<H, \, K\right>$ in this situation. Furthermore, the centralizer $C$ is a lattice automorphism of $\ch(G).$ Thus, if $H \leq \ch(G),$ then $C(H) \leq \ch(G).$ See \cite{CD}, \cite{W}, and \cite{IMI} for details.

We omit the proof of the next theorem, which is essentially the same as that of Lemma 1.1 given in \cite{CD}, except that one uses the properties of the pseudocentralizer along with commutator calculus.

\begin{Theorem}
For any short exact sequence (\ref{shes}), $\pc(G)$ is a lattice with operations $\cap$ and $\left<\;,\;\right>$ and $P$ (the pseudocentralizer) is a lattice automorphism of $\pc(G)$.
\end{Theorem}

\begin{Theorem}\label{inv-image}
For any short exact sequence (\ref{shes}), we have $\pc(G) = q^{-1}(\ch(G_{2})).$
\end{Theorem}

\begin{proof}
Let $H \in \ch(G_{2}).$ If $K \in \pc(G),$ then $m(q(K)) \le m(H)$ and $m_{s}(K) \geq m_{s}\left(q^{-1}(H)\right).$ By Lemma \ref{relation1} and Corollary \ref{relation2}, we have
\[
m_s(K) = m(q(K))|G_{1}||\hbox{im } f \cap K | \le m(H)|G_{1}|^{2} = m_{s}\left(q^{-1}(H)\right).
\]
This means that $m_{s}\left(q^{-1}(H)\right) = m_{s}(K)$ and, thus, $q^{-1}(H) \in \pc(G).$ Therefore, $q^{-1}(\ch(G_{2})) \subseteq \pc(G)$.

Suppose now that $K \in \pc(G).$ We prove that $K \in q^{-1}(\ch(G_{2}))$ by first showing that $q(K)\in \ch(G_{2}).$ If $q(K) \notin \ch(G_{2}),$ then for any $H \in \ch(G_{2})$ we have $m(q(K)) < m(H).$ By Lemma \ref{relation1} and Corollary \ref{relation2} again, we have $m_{s}(q^{-1}(H)) > m_{s}(K).$ This is a contradiction since $m_{s}(K) \ge m_{s}(q^{-1}(H)).$ Thus, $q(K)\in \ch(G_{2}).$
Now, $m_{s}(K) = m_{s}(q^{-1} \circ q)(K)$ because the first part of the proof gives that\\ $(q^{-1}\circ q)(K)\in \pc(G).$ By Lemma \ref{relation1} and Corollary \ref{relation2}, we have
\[
m(q(K)){|G_{1}||\hbox{im }f \cap K|} = m_{s}(K) = m_{s}((q^{-1}\circ q)(K)) = m(q(K))|G_{1}|^{2}.
\]
This means that $|\hbox{im }f \cap K| = |G_{1}|.$ Hence, $\hbox{im }f \le K.$ By Theorem \ref{L1}, $|K| = |G_{1}||q(K)|$ and $|(q^{-1}\circ q)(K)| = |G_{1}||q(K)|.$ The second equality follows from noting that $q \left((q^{-1} \circ q)(K) \right) = q(K)$ and recalling that (\ref{shes}) is exact; so that
$G_{1} = f^{-1} ( (q^{-1} \circ q )(K) \cap \ker q ).$ However, $K \le (q^{-1}\circ q)(K)$ and, consequently, $K = (q^{-1}\circ q)(K).$ Since $q(K) \in \ch(G_{2}),$ we have $K \in q^{-1}(\ch(G_{2})).$
\end{proof}

The first step in computing $m^{\ast}(\h(p^{n}))$ is to find $m^{\ast}(\h(p)).$ This is done using standard group theory.

\begin{Lemma}\label{measure-p}
For any prime $p,$ we have $m^{\ast}(\h(p)) = p^{4}.$
\end{Lemma}

\begin{proof}
Set $Z = Z(\h(p))$ and $\h = \h(p).$ Note that $m(1) = p^{3},$ $m(\h) = p^{4},$ and $m(Z) = \vert Z \vert \vert \h \vert = pp^{3} = p^{4}.$ Let $K \neq Z$ be a proper non-trivial subgroup of $\h.$ Then either $\vert K \vert = p^{2}$ or $\vert K \vert = {p}.$

\vspace{.05in}

\noindent (i) \ If $\vert K \vert = p^{2},$ then $C(K)$ is properly contained in $\h.$ Otherwise, we would have $K \leq C(C(K)) = Z,$ which is impossible. Thus, $m(K) \leq p^{4}.$

\vspace{.05in}

\noindent (ii) \ If $\vert K \vert = p,$ then $C(K)$ is properly contained in $\h$ since $K \neq Z.$ Hence, $m(K) \leq p^{3}.$
\end{proof}

We now arrive at the main theorem of this paper.

\begin{Theorem}\label{MAIN}
For any prime $p,$ we have $m^{\ast}(\hb{n}) = p^{4n}.$
\end{Theorem}

\begin{proof}
We proceed by induction on $n.$ The base step ($n = 1$) follows from Lemma \ref{measure-p}. Suppose that the theorem is true for all values up to $n - 1.$ Since $Z(\hb{n})$ is isomorphic to $\mathbb Z_{p^{n}},$ $m(Z(\hb{n}) = p^{4n}.$ Hence, $m^{\ast}(\hb{n})\ge p^{4n}$.

By Theorem \ref{inv-image}, for any $H \in \ch(\hb{n - 1}),$ we know that
\begin{equation}\label{e:Inverseimage}
m_{s}^{\ast}\left(\hb{n}\right) = m_{s}\left(q^{-1}(H)\right).
\end{equation}
Combining (\ref{e:Inverseimage}), Lemmas \ref{general-p-n} and \ref{compare}, and the induction hypothesis, gives us
\begin{equation} \label{e:PseudoInequality}
m^{\ast}\left(\hb{n}\right) \le m^{\ast}_{s}\left(\hb{n}\right) = m_{s}\left(q^{-1}(H)\right) = p^{6}m(H) = p^{6}p^{4n - 4} = p^{4n + 2}.
\end{equation}
We conclude that $p^{4n} \leq m^{\ast}(\hb{n}) \leq p^{4n + 2}.$ Hence, $m^{\ast}(\hb{n}) = p^{4n + i}$ for some $0 \le i \le 2.$

The remainder of the proof consists of showing that $i = 0.$ Suppose, on the contrary, that $i > 0,$ and let $K \in \ch(\hb{n}).$ By Lemma \ref{p-n} and (\ref{e:PseudoInequality}), we have
\[
p^{4n + 2} \ge m_{s}(K) = p^{3 - \delta(K)}m(K) =  p^{3 - \delta(K)}p^{4n + i} =  p^{4n + 3 + i - \delta(K)}.
\]
This gives $4n + 3 + i - \delta(K) \le 4n + 2;$ that is, $i + 1 \le \delta(K) \le 3.$ Since $i = 1$ or $i = 2,$ we have $\delta(K) = 2$ or $\delta(K) = 3.$

\vspace{.1in}

\noindent \underline{\textbf{Case I:} $\delta(K) = 3$}

\vspace{.05in}

In this case, $K \le Z(\hb{n})$ by Lemma~\ref{emptyI}. Hence, $|K| \leq p^{n}$ and $|C(K)| = p^{3n}.$ This implies that $m(K) \leq p^{4n}.$ Since $K \in \ch(\hb{n})$ and $m^{\ast}(\hb{n}) = p^{4n + i},$ we obtain $m(K) = p^{4n + i} \leq p^{4n}.$ This means that $i = 0$. But this cannot happen as $i > 0$. And so, $\delta(K) \ne 3.$

\vspace{.1in}

\noindent \underline{\textbf{Case II:} $\delta(K) = 2$}

\vspace{.05in}

Since $\delta(K) = 2,$ any injective set of $K$ will have cardinality 1. Furthermore, $i = 1$ and $m(K) = p^{4n + 1}$. By Lemma \ref{p-n}, $m_{s}(K) = p^{3 - 2}m(K) = p^{4n + 2}.$ Hence, $K \in \pc(\hb{n})$ by (\ref{e:PseudoInequality}). This shows that $m_{s}^{*}(\hb{n}) = p^{4n + 2}.$ By Theorem \ref{inv-image}, there exists $\widehat{K} \in \ch\left(\hb{n - 1}\right)$ such that $K = q^{-1}\left(\widehat{K}\right).$ Thus, $K$ has an injective set of cardinality 2 by Lemma~\ref{inj-size}, a contradiction. And so, $i = 0.$ \end{proof}

\begin{Corollary}
For all primes $p,$ we have $\pc(\h(p^{n}))\subseteq \ch(\h(p^{n})).$
\end{Corollary}

\begin{proof}
Let $K \in \pc(\h(p^{n})).$ By Theorem~\ref{inv-image}, $K = q^{-1}(H)$ for some $H \in \ch\left(\h\left(p^{n - 1}\right)\right).$ By Lemma \ref{pre-p-n} and Theorem~\ref{MAIN}, we obtain $m\left(q^{-1}(H)\right) = p^{4n}.$ Thus, $m\left(q^{-1}(H)\right) = m^{\ast}(\h(p^{n}))$ by Theorem \ref{MAIN}.
\end{proof}

\begin{Remark}
The previous Corollary is not always true for groups different from $\h(p^{n}).$ Consider the short exact sequence:
$$1 \to A_{3} \to S_{3} \to \mathbb{Z}_{2} \to 1$$
Then $\ch(S_{3}) = \{A_{3}\}$ and, by Lemma \ref{abelian-case}, $\pc(S_{3}) = \{S_{3}\}.$
\end{Remark}

\section{Appendix}

\noindent \underline{Proof of Theorem~\ref{L1}}

\vspace{.05in}

Let $H_{1} = H \cap \ker q,$ so that $\widehat{H}_{1} = f^{-1}(H_{1}).$ Since $f$ is injective, $\left|\widehat{H}_{1}\right| = |H_{1}|.$ For each $y \in \widehat{H}_{2},$ fix an element $x_{y}$ in $q^{-1}(y)$ if $y \neq 1,$ and let $x_{1} = 1.$ Put $H_{2} = \{x_{y} \, | \, y \in \widehat{H}_{2}\},$ and observe that $H_2$ is merely a subset of $G$ and not a subgroup in general. It is clear that $|H_{2}| = \left|\widehat{H}_{2}\right|.$ We claim that $H_{1}\cap H_{2} = \{1\}$. To see this, let $x \in H_{1} \cap H_{2}$. Since $x \in H_{1}=H \cap \ker q$, then $q(x) = 1.$ Now, assume $x \neq 1.$ Then $x \in H_{2}$ implies that $x=x_{y} \in q^{-1}(y)$ for some $y \in \widehat{H}_{2} - \{ 1\}$. This means that $q(x) = y \neq 1$, a contradiction.

We define a set map $\psi:\widehat{H}_{1}\times \widehat{H}_{2}\to  H$ as follows:
\[
\psi\left(\widehat{h}_{1},\widehat{h}_{2}\right) = f\left(\widehat{h_{1}}\right)x_{\widehat{h}_{2}}
\]
It is easy to see that $\psi$ is well-defined. We prove that $\psi$ is a bijection. Suppose that $\psi\left(\widehat{h}_{1},\widehat{h}_{2}\right) = \psi\left(\widehat{h}^{\prime}_{1},\widehat{h}^{\prime}_{2}\right).$ Then $f\left(\widehat{h}_{1}\right)x_{\widehat{h}_{2}} = f\left(\widehat{h}^{\prime}_{1}\right)x_{\widehat{h}^{\prime}_{2}}.$ This means $f\left(\left(\widehat{h}^{\prime}_{1}\right)^{-1}\widehat{h}_{1}\right)x_{\widehat{h}_{2}}=x_{\widehat{h}^{\prime}_{2}}$ and, thus,
\[
\widehat{h}_{2} = q\left(f\left(\left(\widehat{h}^{\prime}_{1}\right)^{-1}\widehat{h}_{1}\right)x_{\widehat{h}_{2}}\right) = q\left(x_{\widehat{h}^{\prime}_{2}}\right) = \widehat{h}^{\prime}_{2}.
\]
This gives $f\left(\widehat{h}_{1}\right)x_{\widehat{h}_{2}} = f\left(\widehat{h}^{\prime}_{1}\right)x_{\widehat{h}_{2}};$ that is, $f\left(\widehat{h}_{1}\right) = f\left(\widehat{h}^{\prime}_{1}\right).$ Since $f$ is injective, $\widehat{h}_{1}=\widehat{h}_{1}^{\prime}$, thus confirming that $\psi$ is injective.

Now let $h\in H$. If $q(h) = 1,$ then $h \in H_1 = H \cap \ker q,$ and there is an $\widehat{h}\in \widehat{H_{1}}$ such that $f\left(\widehat{h}\right) = h$ and, thus, $\psi\left(\widehat{h}, \, 1\right) = h.$ Suppose next that $q(h)=\widehat{h}_{2}\ne 1,$ and note that $h x_{\widehat{h}_2}^{-1} \in H.$ In fact, since $q(h) = \widehat{h}_2 = q\left(x_{\widehat{h}_2}\right),$ it follows that $ h x_{\widehat{h}_2}^{-1} \in \ker q \cap H = H_1.$ Now, the restriction of $f$ to $\widehat{H}_1$ is a bijection between $\widehat{H}_1$ and $H_{1}$, so there exists a unique $\widehat{h}_{1} \in \widehat{H}_{1}$ such that $f\left(\widehat{h}_{1}\right) = h x_{\widehat{h}_2}^{-1}$. It follows that $\psi\left(\widehat{h}_{1}, \widehat{h}_{2}\right) = f\left(\widehat{h}_1\right)x_{\widehat{h}_2} = h.$ And so, $\psi$ is a bijection. This completes the proof.

\end{document}